\documentclass[11pt]{amsart}
\usepackage[latin1]{inputenc}\usepackage{amsfonts}
\usepackage{amsmath}
\usepackage{amssymb}
\usepackage{amsthm}
\usepackage[T1]{fontenc}
\usepackage{mathrsfs}
\usepackage[dvips]{hyperref}
\usepackage{enumerate}
\usepackage{amssymb}
\usepackage{amsmath}
\usepackage{pstricks}

\usepackage{graphicx}
\newtheorem{lemma}{Lemma}[section]

\newtheorem{theoreme}{Theorem}[section]
\newtheorem{prop}{Proposition}[section]
\newtheorem{coro}{Corollary}[section]
\newtheorem{conjecture}{Conjecture}[section]
\newtheorem{definition}{Definition}[section]
\theoremstyle{definition}
\newtheorem{rem}{Remark}[section]
\newtheorem{ex}{Example}[section]
\usepackage{graphicx}

\def\F{\mathbb{F}}
\def\N{\mathbb{N}}

\def\a{\textbf{a}}
\def\b{\textbf{b}}
\def\p{\textbf{$\mathfrak{p}_q$}}
\def\s{\textbf{$\mathfrak{p}_2$}}

\begin{document}
%\chead{Alina Firicel} 

%%b
\title[Subword complexity and Laurent series] 
{Subword complexity and Laurent series with coefficients in a finite field }

%\thanks{This paper will represent a  portion of the author's ``m�moire de th�se'' at the Universit� Claude Bernard Lyon, thesis advisor Boris Adamczewski.}

\author{Alina Firicel}

\address{Universit\'{e} de Lyon\\
Universit\'{e} Lyon 1 \\
 Institut Camille Jordan \\
  UMR 5208 du CNRS \\
   43, boulevard du 11 novembre 1918 \\ 
   F-69622 Villeurbanne Cedex, France}
   
\email{firicel@math.univ-lyon1.fr}

\maketitle 

\begin{abstract} 
Decimal expansions of classical constants such as $\sqrt2$, $\pi$ and $\zeta(3)$ have long been a source of difficult questions. In the case of  Laurent series with coefficients in
a finite field, where no carry-over difficulties appear, the
situation seems to be simplified and drastically different. On the other hand, Carlitz introduced  analogs of real numbers such as $\pi$, $e$ or
$\zeta(3)$. Hence, it became reasonable to enquire how ``complex''
the Laurent representation of these ``numbers'' is. 

In this paper we prove that the inverse of Carlitz's analog of $\pi$, $\Pi_q$, has in general a linear  complexity, except  in the case $q=2$, when the complexity is quadratic. In particular, this implies the transcendence of $\Pi_2$ over $\F_2(T)$. In the second part, we consider the classes of Laurent series of at most polynomial complexity and of zero entropy. We show that these satisfy some nice closure properties.  \end{abstract}

\section{Introduction and motivations}

A long standing open question  concerns the digits of the real number $\pi=3.14159\cdots$. The decimal expansion of $\pi$ has been calculated to billions of digits and unfortunately, there are no evident patterns  occurring. Actually, for any $b\geq 2$, the $b$-ary expansion of $\pi$ looks like a random sequence (see for instance \cite{BBP}). 
%%b 
More concretely, it is widely believed that $\pi$ is normal, meaning that all  blocks of digits of equal length occur in the $b$-ary  representation of $\pi$ with the same frequency, but current knowledge on this point is scarce. 

 A usual way to describe the disorder of an infinite sequence $\a=a_0a_1a_2\cdots$ is to compute its subword complexity, which is the function that associates to each positive integer $m$ the number $p(\a,m)$  of distinct blocks of
length $m$ occurring in the word $\a$. Let  $\alpha$ be a real number and let $\a$ be the representation of $\alpha$ in an integral base $b\geq 2$. The complexity function of $\alpha$ is defined as follows:
$$p(\alpha, b, m)=p(\a,m),$$
for any positive integer $m$.

Notice that  $\pi$ being normal would imply that its complexity  must be maximal, that is $p(\pi, b, m)=b^m$. In this direction, similar questions have been asked  about other well-known constants like $e$, $\log 2$, $\zeta(3)$   or $\sqrt 2$ and it is widely believed that the following conjecture is true.
\begin{conjecture}\label{maximal}  Let $\alpha$ be one of the classical constants: $\pi$, $e$, $\log 2$, $\zeta(3)$   and  $\sqrt 2$. The complexity of the  real number $\alpha$ satisfies:
$$p(\alpha, b,m)=b^m,$$
for every positive integer $m$ and every  base $b\geq 2$.
\end{conjecture}

\medskip

We mention that in all this paper we will use Landau's notations.  We write $f(m) = \Theta(g(m))$ if there exist positive real numbers $k_1,k_2,n_0 $  such that, for every $n>n_0$ we have $$k_1\left|g(n)\right|< \vert f(n) \vert <k_2\left|g(n)\right|.$$
%%b 
We write also $f(m) = O(g(m))$  if  there exist two positive real numbers $k, n_0$  such that, for every $n\geq n_0$ we have: $$\left|f(n)\right| <k \left|g(n)\right|.$$ 

\medskip
If $\alpha$ is a rational real number  then $p(\alpha,b, m) =O(1)$, for every integer $b\geq 2$. Moreover, there is a classical theorem of Morse and Hedlund \cite{Morse_Hedlund} which states that an infinite sequence $\a=(a_n)_{n\geq 0}$ is eventually periodic if and only if   $p(\a,m)$ is bounded. If not, the complexity function  is strictly increasing. In particular,
\begin{equation} \label{morse_hed}p(\a,m) \geq m+1,\end{equation}
for every nonnegative integer $m$.

A  sequence  which saturates the inequality above is called a Sturmian sequence (see the original papers of Morse and Hedlund \cite{Morse_Hedlund, Morse_Hedlund1}). 

According to this theorem, an irrational real number $\alpha$ has a complexity function which satisfies $ p(\alpha, b, m) \geq m+1$, for every $m\in \N$.  
Concerning irrational algebraic numbers,  the main result known to date in this direction is due to Adamczewski and Bugeaud  \cite{Adamczewski_Bugeaud}. 
%%b
These authors proved that  the complexity of  an irrational algebraic real number  $\alpha$ satisfies 
$$ \lim_{m \rightarrow \infty} \frac{p(\alpha, b, m)}{m}=+\infty, $$  for any base $b \geq 2$.

For more details about  complexity of  algebraic real numbers, see  \cite{Adamczewski_Bugeaud, Adamczewski_Bugeaud_luca}. 
%%b 
For classical transcendental  constants, there is a more ambiguous situation and, to the best of our knowledge, the only result that improves  the bound following from Inequality (\ref{morse_hed})  was  recently proved in \cite{adamczewski}. It concerns the real number $e$ and some other exponential periods. More precisely, Adamczewski showed that if $\xi$ is an irrational real number whose irrational exponent 
$\mu(\xi)=2$, then 
$$ \lim_{m \rightarrow \infty} p(\xi, b, m)-m=+\infty,$$
for any base $b\geq 2$.

\bigskip

The present paper is motivated by this type of questions, but asked  for  Laurent  series with coefficients in a finite field. In the sequel we will denote respectively by $\F_q(T)$, $\F_q[[T^{-1}]]$ and $\F_q((T^{-1}))$ the field of rational functions, the ring of formal series and the field of Laurent series over the finite field 
%%b 
$\F_q$, $q$ being a power of a prime number $p$.

\medskip

 Let us  also recall the well-known analogy between integers, rationals and real numbers on one side, and polynomials, rationals functions, and Laurent series with coefficients in a finite field, on the other side. 
%%b 
 Notice that, the coefficients in $\F_q$ play the role of ``digits'' in the basis given by the powers of the indeterminate $T$. There is still a main  difference: in the case of real numbers, it is hard to control carry-overs when we add or multiply whereas in the case of power series over a finite field, this difficulty disappear.

\medskip
 By analogy with the real numbers, the complexity of a Laurent series is defined as the subword complexity of its sequence of coefficients. Again, the theorem of Morse and Hedlund gives a complete description of the rational Laurent series; more precisely, they are the Laurent  series of bounded complexity. Hence,  most interesting questions concern  irrational  series.

There is a remarkable theorem of Christol \cite{Christol} which describes precisely the algebraic Laurent series over $\F_q(T)$ as follows. 
 Let $f(T)=\sum_{n\geq -n_0} {a_nT^{-n}}$ be a Laurent series with coefficients in   $\mathbb{F}_q$. Then $f$ is algebraic 
 over  $\mathbb{F}_q(T)$ if, and only if, the sequence of coefficients $(a_n)_{n\geq 0}$ is $p$-automatic.

For more references on automatic sequences, see for example \cite{Allouche_Shallit}.  
Furthermore, Cobham proved that the subword complexity of an automatic sequence is at most linear \cite{Cobham}. Hence, an easy consequence of those two results is the following.

\begin{theoreme}\label{algebraic} Let $f\in \F_q((T^{-1}))$ algebraic over $\F_q(T)$. Then we have:
$$p(f,m)=O(m).$$
\end{theoreme}

The reciprocal is obviously not true, since there  are uncountable many Laurent series with linear complexity. 
%%b 
In contrast with real numbers, the situation is thus clarified in the case of algebraic Laurent series.  Also, notice that Conjecture \ref{maximal} and  Theorem \ref{algebraic}  point out the fact that the situations in $\F_q((T^{-1}))$ and in $\mathbb R$ appear to be 
%%b 
completely opposite.

\medskip

On the other hand, 
%the theory of Drinfeld modules provides analogs of some classical transcendental real numbers. More precisely, in 1935, 
Carlitz introduced \cite{carlitz} functions in positive characteristic by analogy with the number $\pi$,  the Riemann $\zeta$ function, the usual exponential and  the logarithm function.  Many of these  were shown to be transcendental over $\F_q(T)$ (see \cite{ demathan_cherif, goss, thakur, wade, yu}).  In the present paper we focus on  the analog of $\pi$, denoted, for each  $q$, by $\Pi_q$,   and we prove that its inverse has a ``low'' complexity.
More precisely, we will prove in Section 3 the following results.

\begin{theoreme}\label{case2}   Let $q=2$. The complexity of the inverse of $\Pi_q$ satisfies:
 $$p\left(\frac{1}{\Pi_2}, m\right)=\Theta(m^2).$$\end{theoreme}
\begin{theoreme}\label{caseq}
   Let $q \geq 3$.  The complexity of the inverse of $\Pi_q$ satisfies: $$p\left(\frac{1}{\Pi_q},m\right)=\Theta(m).$$
\end{theoreme}

Since any algebraic series has a linear complexity (by Theorem \ref{algebraic}), the following corollary yields.

\begin{coro}\label{pi_2} $\Pi_2$ is transcendental over $\mathbb{F}_2(T)$.\end{coro}

\medskip

The transcendence of $\Pi_q$ over $\F_q(T)$ was first  proved by Wade in 1941 (see \cite{wade}) using an analog of a classical method of transcendence  in zero characteristic. Another proof was given by Yu in 1991 (see \cite{yu}), using the theory of Drinfeld modules. Then, de Mathan and Cherif, in 1993 (see \cite{demathan_cherif}), using tools from  Diophantine approximation, proved a more general result, but in particular their result implied the transcendence of $\Pi_q$.

Christol's theorem has also been used as a combinatorial criterion   in order to prove the transcendence of $\Pi_q$. This is what is usually called  an ``automatic proof''.  The non-automaticity and also the transcendence, was first obtained  by Allouche, in \cite{allouche},  via the so-called $q$-kernel. Notice that our proof of transcendence here is based also by Christol's theorem, but  we obtain the non-automaticity of $\Pi_2$ over $\F_2(T)$ as a consequence of the subword complexity.

\medskip

%%b 
Furthermore, motivated by Theorems \ref{case2}, \ref{caseq} and by  Conjecture \ref{maximal}, we  consider the classes of Laurent series of at most polynomial complexity $\mathcal P$ and of zero entropy $\mathcal Z$ (see Section \ref{closure}), which seem to be good candidates to enjoy some nice closure properties. In particular, we prove the following theorem. 

%that both classes are vector spaces over $\F_q(T)$.

\begin{theoreme}\label{vector space} $\mathcal{P}$  and  $\mathcal{Z}$ are  vector spaces over $\F_q(T)$.
\end{theoreme}
\medskip
%%b 
Another motivation of this work is  the article \cite{beals_thakur} of Beals and Thakur. These authors proposed a classification of Laurent series in function of their space or time complexity. This complexity is in fact a characteristic of the (Turing) machine that computes the coefficient $a_i$, if  $f(T):=\sum_{i}{a_iT^{-i}}$.   They showed that some classes of Laurent series have good algebraic properties (for instance, the class of Laurent series corresponding to  any deterministic space class at least linear form a field). They also place some Carlitz's analogs  in the computational hierarchy.

\bigskip

This paper is organized as follows.  Some definitions and basic notions on combinatorics on words and Laurent  series are recalled in Section 2.  Section 3 is devoted to the study of the Carlitz's analog of $\pi$; we prove Theorems \ref{case2} and \ref{caseq}. In Section 4 we study some closure properties of Laurent series of ``low'' complexity (addition,  Hadamard product, derivative, Cartier operator) and we prove Theorem \ref{vector space}; in particular, this provides a criterion of linear independence over $\F_q(T)$ for two Laurent series in function of their complexity. Finally, we conclude in Section 5 with some remarks concerning the complexity of the Cauchy product of two Laurent series, which seems to be a more difficult problem.

\medskip

\section{Terminologies and basic notions}

In this section, we briefly recall some definitions and well-known results from combinatorics on words. Moreover, we recall some basic notions on algebraic Laurent series.\\

A \emph{word} is a finite, as well as infinite, sequence of symbols (or letters) belonging to a nonempty set $\mathcal A$, called \emph{alphabet}. We usually denote words by juxtaposition of theirs symbols.  \\

Given an alphabet $\mathcal A$, we denote by $\mathcal{ A}^*:= \cup_{k=0}^{\infty} \mathcal {A}^k$ the set of finite  words over $\mathcal A$.  Let $V:=a_0a_{1}\cdots a_{m-1} \in \mathcal{ A}^*$. Then the integer $m$ is the length of $V$ and  is denoted by $\left|V\right|$. The word of length $0$ is the empty word, usually denoted by $\varepsilon$.  We also denote by $\mathcal{A}^m$ the set of all finite words of length $m$ and by $\mathcal{A}^{\N}$ the set of all infinite words over $\mathcal A$. We typically use the uppercase italic letters  $X, Y, Z, U, V, W$ to represent elements of $\mathcal{ A}^*$. We also use  bold lowercase letters $\a, \bf b, \bf c, \bf d, \bf e, \bf f$ to represent infinite words. The elements of $\mathcal{ A}$ are usually denoted by lowercase letters $a,b,c,\cdots$.\\

We say that $V$ is a \emph{factor} (or \emph{subword}) of a finite  word $U$ if there exist some finite words $A$, $B$, possibly empty such that $U=AVB$ and we denote it by $V \triangleleft U$. Otherwise, $V \ntriangleleft U$. We say that $X$ is a \emph{prefix} of $U$, and we denote by $X\prec_p U$ if there exists $Y$ such that $U=XY$. We say that $Y$ is a \emph{suffix} of $U$, and we denote by $Y \prec_s U$ if there exists $X$ such that $U=XY$.\\

Also, we  say that a finite word $V$ is a factor (or subword) of an infinite word $\a=(a_n)_{n\geq 0}$ if there exists a nonnegative integer $j$ such that $V=a_ja_{j+1}\cdots a_{j+m-1}$. The integer $j$ is called an occurence of $V$. \\

Let $U,V,W$ be three finite words over $\mathcal A$, $V$ possibly empty. We denote: $$i(U,V,W):=\{ AVB, \, A\prec_s U, \, B\prec_p W, \, A,B\text{ possibly empty}\},$$ and
 $$i(U,V,W)^+:=\{ AVB, \, A\prec_s U, \, B\prec_p W, \, A,B\text{ nonempty}\}.$$

\medskip

If $n$ is a nonnegative integer, we denote by $U^{n}:=\underbrace{UU\cdots U}_{n \text{ times}}$. We denote also $U^{\infty}:=UU\cdots  $, that is $U$ concatenated (with itself) infinitely many times.
An infinite word $\a$ is \emph{periodic} if there exists a finite word $V$ such that $\a=V^{\infty}$. An infinite word is \emph{eventually periodic} if there exist two finite words $U$ and $V$ such that $\a=UV^{\infty}$.

The fundamental operation  on words is concatenation. Notice that $\mathcal{A}^*$, together with concatenation, form  the  \emph{free monoid} over $\mathcal A$, whose neutral element is the empty word $\varepsilon$.

\subsection{Subword complexity}\label{defini} Let $\a$ be an infinite word over $\mathcal A$. As already mentioned in Introduction,  the \emph{subword complexity} of $\a$ is the function that associates to each $m \in \N$ the number $p(\a,m)$  defined as follows:
$$p(\a,m)=\mathrm{Card}\{(a_j,a_{j+1},\ldots, a_{j+m-1}), \, j\in \N\}.$$
 
 For any word $\a$, $p(\a,0)=1$ since, by convention, the unique word of length $0$  is  the empty word $\varepsilon$. \\ 

For example, let us consider the infinite word $\a=aaa\cdots $, the concatenation of a letter $a$ infinitely many times. It is obvious that $p(\a,m)=1$ for any $m\in \N$. More generally, if $\a$ is eventually periodic, then  its complexity function is bounded.

 On the other side, let us consider the infinite word of Champernowne over the alphabet $\{0, 1,2,3, \ldots,9\}$, $\a:=0123456789101112\cdots$. Notice that $p(\a,m)=10^m$ for every positive integer $m$. 
 
 More generally, one can easily prove that for every $m\in \N$ and for every  word $\a$  over the alphabet $\mathcal A$, we have the following:
$$1\leq p(\a,m)\leq (\mathrm{card}\,\mathcal{A})^m.$$

We  give now an important tool we shall use in general, in order to obtain a bound  of the subword complexity function (for a proof see for example \cite{Allouche_Shallit}):

\begin{lemma}\label{add} Let $\a$ be an infinite word over an alphabet $\mathcal A$. We have the following properties:
\begin{itemize}
\item $p(\a,m)\leq p(\a,m+1) \leq \mathrm{card}\,\mathcal{A} \cdot p(\a, m)$, for every integer $m\geq 0$;

\item $p(\a,m+n)\leq p(\a,m)p(\a,n)$, for all integers $m,n \geq 0$. \end{itemize}
\end{lemma}

\medskip
Let $\F_q$ be the finite field with $q$ elements,  where $q$ is a power of a prime number $p$. \\

In this paper, we are interested in Laurent series with coefficients in  $\F_q$. Let  $n_0 \in \N$ and consider the Laurent series: $$f(T)=\sum_{n=-n_0}^{+\infty}{a_nT^{-n}} \in \F_q((T^{-1})).$$

 Let $m$ be a nonnegative integer. We define \emph{the complexity} of  $f$,  denoted by $p(f,m)$, as being equal to the complexity of the infinite word $\a=(a_n)_{n\geq 0}$.

\medskip
\subsection{Topological entropy} Let $\a$ be an infinite word over an alphabet $  \mathcal A$. \emph{The (topological) entropy} of $\a$ is defined as follows:
$$h(\a)= \lim_{m\rightarrow \infty}\frac{\log p(\a,m)}{m}.$$

The limit exists as an easy consequence of the following property: $p(\a,n+m)\leq p(\a,n)p(\a,m)$, for every $m,n \geq 0$ (which is the second part of the Lemma \ref{add}). If the base of the logarithm is the cardinality of the alphabet then:
$$0\leq h(\a)\leq 1.$$
Notice that, by definition, the ``simpler''  the sequence is, the smaller its entropy is. \\

Let  $n_0 \in \N$ and consider the Laurent series $$f(T)=\sum_{n=-n_0}^{+\infty}{a_nT^{-n}} \in \F_q((T^{-1})).$$
We define  \emph{the entropy} of  $f$, denoted by  $h(f)$, as being equal to the entropy of the infinite word $\a=(a_n)_{n\geq 0}$.

%By theorem \ref{morse}, one can easily see  the  that the complexity of a rational series over $F_q$ is in $O(1)$.

 %Moreover, there is an important result of Christol that characterizes   algebraic formal power series over $\F_q(T)$ from a combinatorial point of view. First of all, recall some definitions about algebraic  power series.

%\emph{Automatic sequences.} A sequence $\a=(a_n)_{n \geq 0}$ is said to be $q$-automatic if there exists a $q$-automaton (a machine of a finite number of states) that generates $a_n$ where the entry is the $q$-ary expansion of $n$, for all $n\geq 0$. 

\medskip

\subsection{Morphisms}\label{morphisms}

Let $\mathcal A$ (respectively $\mathcal B$) be an alphabet and let $\mathcal{A}^*$ (respectively $\mathcal{B}^*$) be the corresponding free monoid. A morphism  $\sigma$ is a map from $\mathcal{A}^*$ to $\mathcal{B}^*$ such that $\sigma(UV)=\sigma(U)\sigma(V)$ for all words $U,V \in \mathcal{A}^*$. Since the concatenation is preserved, it is then possible  to define a morphism on $\mathcal A$.

If $\mathcal {A}= \mathcal{B}$ we can iterate the application of $\sigma$. Hence, if $a\in \mathcal A$, $\sigma^0(a)=a$, $\sigma^i(a)=\sigma(\sigma^{i-1}(a))$, for every $i\geq 1$.

Let $\sigma: \mathcal{A}\rightarrow \mathcal{A}$ be a morphism.   The set $\mathcal A^* \cup \mathcal A ^{\N}$ is endowed with a natural topology. 
%%b 
Roughly, two words are close if they have a long common prefix. 
We can thus extend the action of a morphism by continuity to $\mathcal A^* \cup \mathcal A ^{\N}$. 
Then, a word $\a\in \mathcal{A}^{\N}$ is \emph{a fixed point} of a morphism $\sigma$ if $\sigma(\a)=\a$.\\

%Let $U:=U_0U_1\cdots$ and $V:=V_0V_1\cdots$ be two distinct words belonging to $\mathcal A^* \cup \mathcal A ^{\N}$. We define a distance as follows:$$d(U,V)=2^{-\min\{n\in \N, \, U_n \neq V_n\}}.$$With respect to this distance, $\mathcal A^{\N}\cup \mathcal A^*$ becomes a topological space, usually called the \emph{Cantor space}.\\

A morphism $\sigma$ is \emph{prolongable} on $a\in \mathcal{A}$ if $\sigma (a)=ax$, for some $x\in \mathcal{A}^{+}:=\mathcal{A}^*\backslash \{\varepsilon\}$. If $\sigma$ is prolongable then the sequence  $(\sigma^i(a))_{i\geq 0}$ converges to the infinite word: $\sigma^{\infty}(a)=\lim_{i\rightarrow \infty}\sigma^i(a)=ax\sigma(x)\sigma^2(x)\sigma^3(x)\cdots$.

\begin{ex}The Fibonacci word $\textbf {f}=0100101001001\cdots$ is an example of an infinite word generated by iterating the morphism: $\sigma(0)=01$ and $\sigma(1)=0$. More precisely, $\textbf {f}=\sigma^{\infty}(0)$ is  the unique fixed point of $\sigma$.
\end{ex}

The \emph{order of growth} of a letter $x$ is the function $\left|\sigma^n(x)\right|$, for $n\geq 0$. In general, this function   is bounded or, if not,  is growing asymptotically like the function $n^{a_x}b_x^n$. A morphism is said to be \emph{polynomially diverging} if there exists $b>1$ such that, for any letter $x$, the order of growth of $x$ is $n^{a_x}b^n$ and $a_x\geq 1$ for some $x$. A morphism is \emph{exponentially diverging} if every letter $x$ has the order of growth  $n^{a_x}b_x^n$ with $b_x > 1$ and not all $b_x$ are equal. For more details the reader may refer to \cite{Pansiot}.

 A morphism $\sigma$ is said to be uniform of length $m\geq 2$ if $\left|g(x)\right|=m $. Notice that a word 
 %%b Fais attention ! Tu changes modulus en length et la ligne d'apres tu remets modulus.
 generated by an uniform morphism of length $m$ is $m$-automatic (see for example \cite{Cobham}). In particular, its complexity is $O(1)$ if the word is eventually periodic; otherwise, it is $\Theta (m)$.

For more about the complexity function of words generated by morphisms there is a classical theorem of Pansiot  \cite{Pansiot} that characterizes the asymptotic behavior of factor complexity of  words obtained by iterating a morphism.

\medskip

%\subsection{Sturmian sequences\label{sturmian}

\medskip

\subsection{Algebraic Laurent series} A Laurent series $f(T)=\sum_{n\geq -n_0}{a_nT^{-n}}\in \F_q((T^{-1}))$ is said to be \emph{algebraic} over the field $\F_q(T)$ if there exist an integer $d\geq 1$ and polynomials $A_0(T), A_1(T), \ldots, A_d(T)$, with coefficients in  $\F_q$ and not all zero, such that:
$$A_0+A_1f+\cdots+A_df^d=0.$$
Otherwise, $f$ is transcendental over $\F_q(T)$.\\

Let us now give an example of Laurent series algebraic over the field of rational functions. 

\begin{ex} Let us consider  the formal  series $f(T)=\sum_{n\geq 0}{c_nT^{-n}}\in \F_3[[T^{-1}]]$ where $\textbf{c}:=(c_n)_{n\geq 0}$ is the Cantor sequence defined as follows:
$$c_n=\begin{cases} 1 \text { if }  (n)_3   \text{ contains only }  0  \text{ and } 2;\\ 0  \text{ if }  (n)_3\text{ contains the letter }1.\end{cases}.$$

Here $(n)_3$ denotes the expansion in base $3$ of $n$. By definition, we get that $c_{3n}=c_n=c_{3n+2}$ and $c_{3n+1}=0$, for all $n\in \N$.

 We have:
\begin{eqnarray*}
f(T)&=&\sum_{n\geq0} {c_{3n} T^{-3n}}+ \sum_{n\geq0} {c_{3n+1} T^{-3n-1}}+\sum_{n\geq0} {c_{3n+2} T^{-3n-2}}\\
   &=&\sum_{n\geq0} {c_{n} T^{-3n}}+\sum_{n\geq0} {c_{n} T^{-3n-2}}.\end{eqnarray*}
  
Hence, $$f(T)=f(T^3)+T^{-2}f(T^3)$$ and, since we are in characteristic $3$, we  obtain that $f$ satisfies the following equation:  $$(1+T^{2})f^2(T)-T^2=0.$$
 Thus $f$ is an algebraic Laurent series.

Notice also that, the infinite sequence $\textbf{c}$ is $3$-automatic, as predicted by Christol's theorem, and in particular the complexity of $\textbf c$ satisfies:
$$p(\textbf c, m)=O(m).$$
\end{ex}
%It is easy to prove that the sequence $\bf c$ is $3$-automatic, since there is the following automaton generating it:
%\begin{figure}[h]  \centering\includegraphics[scale=0.65]{cantor.pdf}
%\caption{Automate engendrant la suite des entiers de Cantor \label{automatecantor} \end{figure}By definition, one could easily see that $c_{3n}=c_n=c_{3n+2}$ and $c_{3n+1}=0$.\\\ We shall prove next that the formal power series $C(T)$ is algebraic over $\F_3(T)$:\begin{eqnarray*}C(T)&=&\sum_{n\geq0} {c_{3n} T^{-3n}}+ \sum_{n\geq0} {c_{3n+1} T^{-3n-1}}+\sum_{n\geq0} {c_{3n+2} T^{-3n-2}}\ &=&\sum_{n\geq0} {c_{n} T^{-3n}}+\sum_{n\geq0} {c_{n} T^{-3n-2}}.\end{eqnarray*}\ Hence, $$C(T)=C(T^3)+T^{-2}C(T^3)$$ and using the Frobenius map we finally obtain:  $$(1+T^{2})C(T)^2-T^2=0.$$\bigskip
 
% The example above is actually a particular case of a general property of algebraic formal power series over $\F_q(T)$. Indeed, Christol's theorem gives a combinatorial characterization of algebraic formal power series:

\section {An analogue of  $\Pi$}

  In 1935, Carlitz \cite{carlitz} introduced  for function fields in positive characteristic an analog of the exponential function defined over $\mathcal C_{\infty}$,  which is the completion of the algebraic closure of $\F_q((T^{-1}))$ (this is the natural analogue of the complex numbers field). In order to get good properties in parallel  with the complex exponential,  the resulting analogue, $z\rightarrow e_C(z)$, satisfies:
$$e_C(0)=0,  \, d/dz(e_C(z))=1 \text{ and } e_C(Tz)=Te_C(z)+e_C(z)^q.$$

This is what we call the Carlitz exponential and the action $u \rightarrow Tu+u^q$ leads to the definition of the Carlitz $\F_q[T]$-module, which is in fact a particular case of Drinfeld module.
The Carlitz exponential, $e_C(z)$, may be defined  by the following infinite product:
$$e_C(z)=z \prod_{a\in \F_q[T],\, a\neq 0}{(1-\frac{z}{a\widetilde {\Pi}_q})}$$
 where $$\widetilde{\Pi}_q=(-T)^{\frac{q}{q-1}}\prod_{j=1}^{\infty}{\left(1-\frac{1}{T^{q^{j}-1}}\right)^{-1}}.$$

 Since $e^z=1$ if and only if $z \in 2 \pi i \mathbb Z$ and since $e_C(z)$ was constructed by analogy  such that $e_C(z)=0$ if and only if $z \in \widetilde {\Pi}_q \F_q[T]$ (in other words the kernel of $e_C(z)$ is  $\widetilde{\Pi}_q \F_q[T]$), we get a good analogue $\widetilde{\Pi}_q$ of $2\pi i$.
In order to obtain a good analogue of the real number $\pi$,  we take its one unit part and hence we obtain:
  $$\Pi_q=\prod_{j=1}^{\infty}{\left(1-\frac{1}{T^{q^{j}-1}}\right)^{-1}}.$$
  
  \medskip
For more details about analogs given by the theory of Carlitz modules, and in particular about the exponential function or its fundamental period $\widetilde{\Pi}_q$, we refer the reader to the monographs \cite{goss, thakur}.
\medskip

If we look for the Laurent series expansion of $\Pi_q$, then we obtain that
\begin{equation}\nonumber
\Pi_q=\prod_{j=1}^{\infty}{\left(1-\frac{1}{T^{q^{j}-1}}\right)^{-1}}=\sum_{n\geq 0}{a_nT^{-n}},
\end{equation}
where $a_n$ is defined as the number of partitions of  $n$  whose parts take values in  $I=\{q^j-1,j\geq 1\}$, taken modulo $p$.

To compute the complexity of $\Pi_q$, we would like to find a closed formula or some recurrence relations for the sequence of partitions $(a_n)_{n\geq 0}$. This question seems quite difficult and we are not able to solve it at this moment. 

\medskip

However, it was shown  in \cite{allouche} that the inverse of  $\Pi_q$ has  the following  simple Laurent series expansion:

$$\frac{1}{\Pi_q}=\prod_{j=1}^{\infty}\left(1-\frac{1}{X^{q^j-1}}\right)=\sum_{n=0}^{\infty}{p_n X^{-n}}$$
where the sequence  $\p=(p(n))_{n\geq 0}$ is defined as follows:
\begin{equation}\label{definition} p_n=\begin{cases} 1 \text{ if } n=0;\\
 (-1)^{\mathrm{card}\, J} \text{ if there exists a set } J \subset \mathbb N^*  \text{ such that } n= \sum_{j\in J }{(q^j-1)}; \\
0 \text{ if there is no set } J \subset \mathbb N^*  \text{ such that } n= \sum_{j\in J }{(q^j-1)}. \end{cases}\end{equation}
We mention that if such a decomposition exists, it is  unique.  
\bigskip

%%b 
In the rest of this section we  will prove Theorem \ref{case2} and \ref{caseq}.

\subsection{Proof of Theorem \ref{case2}}

In this part we  study the  sequence  $\s=(p_n^{(2)})_{n\geq 0}$, defined  by the  formula (\ref{definition}) in the case where $q=2$. More precisely:
\begin{equation}\label{definit}p_n^{(2)}= \begin{cases} 1 \text{ if } n=0 \text{ or if there exists  } J \subset \mathbb N^*  \text{ such that } n= \sum_{j\in J }{(2^j-1)}; \\
0 \text{ otherwise}.
\end{cases}\end{equation}
 
In order to lighten the notations, in the rest of this subsection we set $p_n:=p_n^{(2)}$ so that $\s=p_0p_1p_2\cdots$.

For every $n \geq 1$, we denote by  $W_n$ the factor of $\s$ that occurs between positions $2^n-1$ and $2^{n+1}-2$, that is:  
$$W_n:=p_{2^n-1}\cdots p_{2^{n+1}-2}.$$
       We also set $W_0:=1$.
   Observe that $\left|W_n\right|=2^n$.\\
With these notations the infinite word $\s$ can be factorized as:
       $$\s=\underbrace{1}_{W_0}\underbrace{10}_{W_1}\underbrace{1  1 00 }_{W_2}\underbrace{11011000}_{W_3}\cdots=W_0W_1W_2\cdots.$$
         
         In order to compute the complexity function of $\s$, we need the following lemmas.
         
         \begin{lemma}\label{l1}Let $n$ and $k$ be two positive integers such that: $k< 2^n-1$. Then
        $k$ can be written as  $\sum_{j\in J }{(2^j-1)}$ if and only if $k+(2^n-1)$ can be written as $\sum_{i\in I }{(2^i-1)}$, where $I$ and $J$ are finite subsets of $\mathbb{N}^*$.\end{lemma}

 \begin{rem} This is equivalent to say that $a_k=1$ if and only if  $a_{k+(2^n-1)}=1$.
\end{rem}
\begin{proof} The first part is pretty obvious. If  $k=\sum_{j\in J }{(2^j-1)}$, then: \[k+(2^n-1)=\sum_{j\in J }{(2^j-1)+(2^n-1)}=\sum_{j\in J\cup \{n\}}{(2^j-1)}.\]
For the second part, let suppose  that  $k+(2^n-1)=\sum_{i\in I }{(2^i-1)}$.  We prove that $k$ can be also represented in this form. More precisely, we show that $n\in I$ and consequently $k= \sum_{i\in I \setminus \{n\} }{(2^i-1)}$.

Notice that  $I$ cannot contain any index greater than $n$ since $k<2^n-1$. We argue by contradiction and we assume  that $I$ does not contain $n$.  Then  $\sum_{i\in I }{(2^i-1)} < 2^n-1$ since 
 \[ \sum_{i\in I }{(2^i-1)}\leq \sum_{i=1}^{ n-1 }{(2^i-1)}= 2(2^{n-1}-1)-(n-1)=2^n-n-1 < 2^n-1.\]
 
    Hence,  if $n$ does not belong to $I$, then $k+(2^n-1) < 2^n-1$ which is absurd. Consequently,  $n$  belongs to $I$  and thus $k= \sum_{i\in I \setminus \{n\} }{(2^i-1)}$.
\end{proof}

\bigskip
       
  \begin{lemma} \label{concaten}   
        For every $n\geq2$ we have  $W_n=1W_1W_2\cdots W_{n-1}0$.\end{lemma}
        
      \begin{proof}   Clearly, the word $W_n$  begins with $1$ because $p_{2^n-1}=1$. The word $W_n$ ends with $0$ since the last letter is $p_{2^{n+1}-2}=0$.
       To explain the structure of $W_n$, we  split the word $W_n$ into subwords as follows:
      
     $$W_n=\underbrace{p_{2^n-1}}_{1} \underbrace{p_{(2^n-1)+(2-1)}p_{(2^n-1)+(2^2-2)}}_{W'_1}\underbrace{p_{(2^n-1)+(2^2-1)}\cdots p_{(2^n-1)+(2^3-2)}}_{W'_2}$$
      $$\cdots \underbrace{p_{(2^n-1)+(2^{n-1}-1)}\cdots p_{(2^n-1)+(2^n-2)}}_{W'_{n-1}} \underbrace{p_{2^{n+1}-2}}_{0}.$$
      
      Since by Lemma \ref{l1}  $p_{(2^n-1)+k}=p_k$ for $k< 2^n-1$, we obtain that $W'_i=W_i$, for $1\leq i \leq n-1$.
     \end{proof}
        
Since the subword $W_n$ ends with $0$, we can define $U_n$ by $W_n:=U_n0$, for every $n \geq 1$. Thus,  $U_1=1$, $U_2=110$.

    \begin{lemma}\label{U_n}    For every $n\geq1$, we have  $U_{n+1}=U_nU_n0$.
      \end{lemma}
        
      \begin{proof}  By Lemma \ref{concaten}, $U_n=1W_1W_2\cdots W_{n-1}$ for all $n\geq 2$.  Consequently:
       $$U_{n+1}=\underbrace{1W_1W_2\cdots W_{n-1}}_{U_n}W_n=U_nW_n=U_n\underbrace{U_n0}_{W_n}.$$
     \end{proof}

   %The infinite word $\s$ can be factorized as:$$\s=U_n\underbrace{ U_n0}_{w_n} \underbrace{U_nU_n00 }_{w_{n+1}}\underbrace{U_nU_n0U_nU_n000}_{w_{n+2}} \cdots.$$ Moreover, $V_n=0w_1 w_2\cdots w_{n-1}$. It is different from $U_n$ only by the first letter. Up to  one word,  we may assume that $V_n=U_n$ and so the word $\s$ may be written  as follows:$$\s=U_n\underbrace{ U_n0}_{w_n} \underbrace{U_nU_n00 }_{w_{n+1}}\underbrace{U_nU_n0U_nU_n000}_{w_{n+2}} \cdots.$$

  \begin{lemma}\label{plage}

 For every $n\geq 2$, there exists a word $Z_n$ such that $W_n=1Z_n10^n$ and   $0^n \ntriangleleft Z_n$ (in other words  $W_n$ ends with exactly  $n$ zeros and $Z_n$ does not contain blocks of $0$ of length larger than $n-1$). This is equivalent to say that  $U_n=1Z_n1 0^{n-1}$ and $0^n \ntriangleleft Z_n$.
\end{lemma}

\begin{proof} We  argue by induction on $n$.

 %For  $n=1$: $W_1=10$. It ends with only one zero and  obviously there are no other zeros. 
 
For $n \geq 2$, $W_2= 1100$ ends with two zeros and obviously there are no other zeros.  

We assume that  $W_n$ ends with  $n$ zeros  and  does not contain other block of zeros of length greater than $n-1$. We  show this statement holds for  $n+1$.
By Lemma \ref{U_n} $$W_{n+1}=U_{n+1}0=U_nU_n00.$$
 
As  $U_n$ ends exactly with  $n-1$ zeros (by induction hypothesis), then $W_{n+1}$ ends also by $n+1$ zeros. Since we have that $U_n=1Z_n1 0^{n-1}$ and  $0^n \ntriangleleft Z_n$ then $W_{n+1}=1Z_n1 0^{n-1}1Z_n1 0^{n-1}00=1Z_{n+1}0^{n+1}$, when $Z_{n+1}:=Z_n 1 0^{n-1}1 Z_n$. Since $0^n \ntriangleleft Z_n$, then $0^{n+1} \ntriangleleft Z_{n+1}$. 
This completes the proof.
 \end{proof}

  \bigskip

%Ainsi pour chaque $m$ fixe on decompose le mot infini principal $\s$ en fonction de $n$ fixe aussi.\\

\begin{lemma}\label{alphabet} For every $n \geq 1$, let  $A_n:=\{U_n^20^k, \, k \geq 1\}$. Then $\s\in A_n^{\mathbb{N}}$. 
\end{lemma}

\begin{proof} Let $n\geq 1$. By definition of $W_n$ and $U_n$ and by Lemma \ref{concaten}, the infinite word $\s$ can be factorized as: \begin{equation}\label{conc}\s=\underbrace{1W_1W_2\cdots W_{n-1}}_{U_n}\underbrace{ U_n0}_{W_n} \underbrace{U_{n+1}0 }_{W_{n+1}}\underbrace{U_{n+2}0}_{W_{n+2}} \cdots.\end{equation}

  We   prove that for every positive integer $k$, there exist a positive integer $r$ and $k_1, k_2,\ldots, k_r \in \N^*$ such that: 
   \begin{equation}\label{U_n+k} U_{n+k}=U_n^20^{k_1}U_n^20^{k_2}\cdots U_n^20^{k_r}.\end{equation}
   
 We argue by induction on $k$. For $k=1$, we have  $U_{n+1}=U_nU_n0=U_n^20$. We suppose that the relation  (\ref{U_n+k}) is true for $k$ and we  show it for  $k+1$. By Lemma \ref{U_n}: $$U_{n+k+1}=U_{n+k}U_{n+k}0=U_n^20^{k_1}U_n^20^{k_2}\cdots U_n^20^{k_r}U_n^20^{k_1}U_n^20^{k_2}\cdots U_n^20^{k_r+1}.$$ 
By equality (\ref{conc}), this ends the proof.
\end{proof}

\bigskip
Fix $m\in \N$. Then,  there is a unique integer $n$ such that: 
\begin{equation}\label{n_m} 2^{n-1}< m\leq 2^n.
\end{equation}

\begin{lemma}\label{prefix} Let $m \in \mathbb{N}$. All distinct words of length $m$ of $\s$ occur in the prefix:

$$P_m=W_0W_1\cdots W_m.$$
\end{lemma}

\begin{proof}

Let $m,n$ be some positive integer satisfying  \ref{n_m}.

%Then, obviously, there is a unique integer $n$ such that: \begin{equation}\label{pref}2^{n-1}< m\leq 2^n.\end{equation}

%Also, by  lemma \ref{concaten}, $U_n=1W_1\cdots W_{n-1}$.
We show that all distinct words of length $m$ occur in the prefix $$P_m=W_0W_1W_2\cdots W_{n-1}W_n\cdots W_m=U_n \underbrace{U_n0}_{W_n}\underbrace{U_n U_n00 }_{W_{n+1}}\underbrace{U_n U_n0U_n U_n000}_{W_{n+2}} \cdots W_{m};$$
the second identity follows by Lemmas \ref{concaten} and \ref{U_n}.

Notice that we have to consider all the words till $W_m$ because the word $0^m$ first occurs in $W_m$.

Also, by Lemma \ref{plage} and using the identity (\ref{U_n+k}), $W_i$ ends with $U_n U_n0^{i-n+1}$, for every $i\geq n+1$.  Consequently, all the words $U_n U_n0^k$, $ 0\leq k \leq m-n+1$ are factors of $P_m$.

 Moreover, notice that if $B_n:=\{U_n U_n0^k, 0\leq k \leq m-n+1 \}$, then $P_m \in B_n^*$. This follows from Lemmas \ref{alphabet} and \ref{plage} (there are no blocks of zeros of length  greater than $m$ in $W_0W_1\cdots W_m$).

After the occurrence of $W_m$, it is not possible to see new different subwords  of length $m$.
Indeed, suppose that there exists a word $F$ of length $m$ that occur in $W_{m+1}W_{m+2}W_{m+3}\cdots$ and does not occur in $P_m$. Then, by Lemma \ref{alphabet} and by the remark above,   $F$ must occur in the words $U_n0^k U_n$, with $k\geq m-n+1$. But since $U_n$ ends with  $n-1$ zeros (by Lemma \ref{plage}), $F$ must be equal to $0^m$ or  $0^iP_i$, where $i\geq m-n+2$ and $P_i \prec_p U_n$, or $F$ must occur in $U_n$. But all these words already occur in $P_m$. This contradicts our assumption.
\end{proof}

%Our goal is next to find all distinct words of length $m$ occurring in $\s$. By Lemma \ref{prefix}, it suffices to look in the prefix $$P_m=U_n \underbrace{U_n 0}_{W_n}\underbrace{U_n U_n00 }_{W_{n+1}}\underbrace{U_n U_n0U_n U_n000}_{W_{n+2}} \cdots W_{m}.$$

\subsubsection{An upper  bound for $p( \frac{1}{\Pi_2},m)$}

In this part we  prove the following result.
\begin{prop}\label{upper}  $$p(\s,m)\leq \frac{(m-\log m)(m+\log m+2)}{2}+2m.$$
 \end{prop}

\begin{proof} %Let $m\in \N$. There is an unique $n\in \N$ such that $2^{n-1}<m\leq 2^n $. 

  In order to find  all different factors of length $m$ that occur in $\s$,  it suffices, by Lemmas \ref{alphabet} and \ref{prefix},  to consider factors appearing in the word  $U_n U_n$  and in  the sets  $i(U_n,0^k,U_n)$, where $1\leq k \leq m-n$. \\

In the word $U_n U_n$ we can find at most  $\left|U_n\right|$ distinct words of length $m$. Since $\left|U_n\right|= 2^n-1$ and  $2^{n-1}<m\leq 2^n $,  the number of factors of length $m$ that occur in $U_n U_n$ is  at most $2^{n}$, so at most $2m$.\\

Also, it is not difficult to see that $\left|i(U_n,0^k,U_n)\right| \cap {\mathcal A}^m \leq m-k+1$.
 %, it is possible to find at most $m-k+1$ new different subwords of length $m$, since (the words of length $m$ existing in the prefix of word $U_n$ have already been counted in $U_n U_n$). Indeed, we are interested in all words  of length $m$ of form $\alpha_k0^k \beta_k$, where $\alpha_k$ and $\beta_k$ are two words possible empty (more precisely, $\alpha_k \prec_s U_n$ and $\beta_k\prec _p U_n$). Thus,  we may find $m-k+1$ words that contain the block $0^k$.\\
The total number of subwords occurring in all these sets, for $1\leq k \leq m-n$,  is  less than or equal to:
  $$\sum_{k=1}^{m-n}{(m-k+1)}= (m+1)(m-n)-\frac{(m-n)(m-n+1)}{2}.$$
  
Counting all these words and using the fact that $2^{n-1}<m\leq 2^n $, we obtain that:
$$p(\s,m)\leq  2m+\frac{(m-n)(m+n+1)}{2} < \frac{(m-\log m)(m+\log m+2)}{2}+2m$$ as claimed.
  \end{proof}

  \subsubsection{A lower bound for  $p( \frac{1}{\Pi_2},m)$}

In this part we  prove the following result.

\begin{prop}\label{lower}  $$p(\s,m) \geq\frac{(m-\log m)(m-\log m+ 1)}{2}.$$
\end{prop}

\begin{proof}
 By Lemma \ref{prefix}, we have to look for distinct  words of length $m$ occurring in $W_n W_{n+1}\cdots W_m$.
 
 % More precisely, we shall count the total number of factors of form $\alpha_k0^k\beta_k$ ($\alpha_k$ and $\beta_k$  being two words, possibly empty) that belong to  $i(W_k, \varepsilon,   W_{k+1})$ for every  integer $k\in \left[n ,m-1\right]$. \\                   
 
 In order to prove this proposition, we use the final blocks of $0$ from each $W_i$. These blocks are increasing (as we have shown in Lemma \ref{plage}). First, in the word $W_m$ we  find for the  first time the word of length $m$: $0^m$.\\
 
  In the set $i(W_{m-1}, \varepsilon, W_m)$, we find  two  distinct  words of length $m$ that cannot be seen before ($10^{m-1}$ and $0^{m-1}1$) since there are no other words containing blocks of zeros of length $m-1$ in $i(W_k,\varepsilon, W_{k+1})$, for $k<m-1$. \\
  
  %For $n\leq k \leq m-2$, we count all the words of length $m$ that contain the block $10^k1$.
   More generally, fix $k$ such that $n\leq k \leq m-2$.  Since   $$W_kW_{k+1}= \underbrace{1Z_k10^k}_{W_k}\underbrace{1 Z_{k+1}10^{k+1}}_{W_{k+1}},$$ in $i(W_k,\varepsilon, W_{k+1})$ we find $m-k+1$ words of length  $m$ of form $\alpha_k0^k\beta_k$. More precisely, the words we count here are the following: $S_{m-k-1}10^k$, $S_{m-k-2}10^k1$, $S_{m-k-3}10^k1T_1$,..., $S_10^k1T_{m-k-2}$, $0^k1T_{m-k-1}$, where $S_i\prec_s Z_{k}$ and $T_i\prec_p Z_{k+1}$, $\left|S_i\right|= \left|T_i\right|=i$, for every integer $i$,  $1 \leq i \leq m-k-1$.\\
   
\ All these words cannot be seen before, that is in $i(W_s,\varepsilon, W_{s+1})$, for $s<k$, since there are no blocks of zeros  of length $k$ before the word $W_k$ (according to Lemma \ref{plage}).  
Also,  in  $i(W_s, \varepsilon, W_{s+1})$, for $s>k$, we focus on the words $\alpha_s0^s\beta_s$ and hence they are different from all the words seen before (because $k<s$).\\
  
 Consequently, the total number of subwords of length $m$ of form $\alpha_k0^k\beta_k$ considered before,  is  equal to $$1+2+\ldots + (m-n+1)=\frac{(m-n+1)(m-n+ 2)}{2}.$$\\ 
Since $2^{n-1}<m\leq 2^n$ we obtain the desired lower bound.
\end{proof}
\bigskip

\begin{proof}[Proof of Theorem \ref{case2}] It follows from Propositions \ref{upper} and \ref{lower}. 
\end{proof} 
 
A consequence of Theorem \ref{case2} and Theorem \ref{algebraic} is the following result of transcendence.
 
   \begin{coro}\label{transc} Let $K$ be a finite field and $(p_n^{(2)})_{n\geq 0}$ the sequence defined in (\ref{definit}). Let us  consider the associated formal series over $K$:  $$f(T):=\sum_{n\geq 0}{p_n^{(2)}T^{-n}}\in K[[[T^{-1}]].$$ Then $f$ is transcendental over $K(T)$. \end{coro}
  
  Notice that, if $K=\F_2$ then the formal series $f$ coincide with $1/\Pi_2$ and hence Corollary \ref{transc} implies Corollary \ref{pi_2}.
  
  \begin{rem} In  \cite{allouche_betrema_shallit}, the authors proved that the sequence $\s$ is the fixed point of the morphism $\sigma$ defined by $\sigma(1)=110$ and $\sigma (0)=0$. 
  
  In an unpublished note \cite{allouche_unpublished}, Allouche showed that the complexity of the sequence $\s$ satisfies, for all $m\geq 1$, the following inequality:
  $$p(\s,m)\geq C m\log m,$$
  for some strictly positive constant $C$.
  
  \end{rem}

  \subsection{Proof of Theorem \ref{caseq}} In this part we  study the sequence $\p =(p_n^{(q)})_{n\geq 0}$ defined by the formula (\ref{definition}) in the case where $q \geq 3$. In the following, we will consider the case  $q=p^n$, where $p\geq 3$.
% If $p=2$,  the only difference is that $-1$ has to be identified with $1$. The sequence will be simpler and  consequently the complexity will be lower.

\begin{prop} \label{upper2} Let $q \geq 3$. For every positive integer $m$: $$ p(\p,m)\leq (2q+4)m+ 2q-3.$$
\end{prop}

 In particular, this proves the Theorem \ref{caseq}. Indeed, we do not have to find a lower bound for the  complexity function, as   the sequence $\p$ is not eventually periodic (see Remark \ref{notperiodiq}) and thus, by the inequality (\ref{morse_hed}) we have that:
 $$p(\p, m)\geq m+1,$$
 for any $m\geq 0$.
 
 \medskip
 
 In order to lighten the notations, we  set in the sequel $p_n:=p_n^{(q)}$ so that $\p=p_0p_1p_2\cdots$. 

For every $n \geq 1$, we denote by $W_n$ the factor of $\p$ defined in the following manner:
$$W_n:=p_{q^n-1}\cdots p_{q^{n+1}-2}.$$

        Let us fix $W_0:=0^{q-2}=\underbrace{00\cdots 0}_{q-2}$ and $\alpha_0: = q-2$. Thus $W_0=0^{\alpha _0}$.

 In other words,  $W_n$ is the factor of $\p$ occurring between positions  $q^n-1$ and $q^{n+1}-2$. Notice that  $\left|W_n\right|=q^n(q-1)$.\\
        
       With these notations the infinite word $\p$  may be factorized as follows: $$\p=1\underbrace{00\cdots0}_{W_0}\underbrace{(-1)00\cdots 0}_{W_1}\underbrace{(-1) \cdots 00 1 00\cdots 0}_{W_2}(-1)00\cdots.$$
 
In the following, we  prove some lemmas that we   use in order to bound from  above the complexity function of $\p$.
 \begin{lemma}\label{l2} Let $k$  and $n$ be two positive integers such that: $$k \in \left[2(q^n-1), q^{n+1}-2\right].$$  Then there is no set $J\subset \mathbb{N^*}$ such that  $k= \sum_{j\in J }{(q^j-1)}$. In other words,  $p_k=0$. \end{lemma}
 
 \begin{proof} We argue by contradiction and we assume that there exists a set $J$ such that  $k= \sum_{j\in J }{(q^j-1)}$. Since $k < q^{n+1}-1$, then obviously $J$ must be a subset of $ \{1,2, 3,\ldots, n\}$. Consequently 
 $$k= \sum_{j\in J }{(q^j-1)} \leq \sum_{j=1}^n{(q^j-1)}.$$
 Then
 $$k \leq \sum_{j=1}^n{(q^j-1)} =q\frac{q^{n}-1}{q-1}-n<2(q^n-1)<k$$
 which is absurd.  
 \end{proof}

\begin{lemma}\label{l3} Let $k$ and $n$ be two positive integers such that: $k<q^n-1$.
Then $k$ can be written as $\sum_{j\in J}{(q^j-1)}$ if and only if  $k+(q^n-1)$ can be written as  $\sum_{j\in I}{(q^j-1)}$, where $I$ and $J$ are finite subsets of $\N^*$.
Moreover, $J\cup \{n\}=I$. \end{lemma}
 
 \begin{rem} This is equivalent to say that $p_k=-p_{k+(q^n-1)}$ for  $k$ and  $n$ two positive integers such that $k<q^n-1$.
\end{rem}
 
 \begin{proof} The proof is similar to that  of Lemma \ref{l1} (just replace $2$ by $q$).
 % If $n=\sum_{j\in J }{(q^j-1)}$ then $$n+(q^k-1)=\sum_{j\in J }{(q^j-1)}+(q^k-1)=\sum_{j\in J\cup \{k\}}{(q^j-1)}$$and the first part follows immediately. We will prove now the second part. Assume that $n+(q^k-1)$ can be represented as $\sum_{i\in I }{(q^i-1)}$. Then $n$ is also of this form. More precisely, we  show that  $k\in I$  and hence $n= \sum_{i\in I \setminus \{k\} }{(q^i-1)}$.We assume by contradiction that $I$ does not  contain $k$, then $\sum_{i\in I }{(q^i-1)} < q^k-1$. Indeed, $$ \sum_{i\in I }{(q^i-1)}\leq \sum_{i\leq k-1 }{(q^i-1)}= q\frac{q^{k-1}-1}{q-1}-(k-1) < q^k-1.$$Since $I$ does not contain $k$ than  $n+(q^k-1) < q^k-1$ which is not true. Consequently,  $I$ contains $k$  and then $n= \sum_{i\in I \setminus \{k\} }{(q^i-1)}$.
       \end{proof}
\bigskip

 If $W=a_1a_2\cdots a_l \in \{0,1,-1\}^l$ then  set $\widehat{W}:=(-a_1)(-a_2)\cdots(-a_l)$.

 \begin{lemma}\label{concate} For every  $n\geq 1$ we have the following:
 % Avec ces notations, on peut remarquer que (et d�montrer par r�currence par exemple):
 $$W_n=(-1)\widehat{W_0}\widehat{W_1}\cdots \widehat{W_{n-1}}0^{\alpha_n}$$
 with  $\alpha_n=(q^{n+1}-1)-2(q^n-1)$.
\end{lemma}
\begin{proof}

Obviously, the word $W_n$ begins with $-1$ since $p_{q^n-1}=-1$. In order to prove the relation above it suffices to split 
 $W_n$ into  subwords as follows:

$$W_n=\underbrace{p_{q^n-1}}_{-1}\underbrace{0\cdots 0}_{W'_0} \underbrace{p_{(q^n-1)+(q-1)} p_{(q^n-1)+(q^2-2)}}_{W'_1}\underbrace{p_{(q^n-1)+(q^2-1)}\cdots p_{(q^n-1)+(q^3-2)}}_{W'_2}$$
      
      $$\cdots \underbrace{p_{(q^n-1)+(q^{n-1}-1)}\cdots p_{(q^n-1)+(q^n-2)}}_{W'_{n-1}}\underbrace{p_{2(q^n-1)}\cdots p_{q^{n+1}-2}}_{0^{\alpha_n}}.$$
    
   Since $p_{(q^n-1)+k}=-p_k$, for every $k< q^n-1$ (by Lemma \ref{l3}), we obtain that $W_i'=\widehat{W_i}$, for $0\leq i \leq n-1$. Lemma \ref{l2} ends the proof.
    %Splitting as above  and using the two previous lemmas, we see the occurrence of the word $\widehat{W_i}$, for $i <n$ and next, a  block of zeros because, if the index $k \in \left[2(q^n-1), q^{n+1}-2\sight]$ then $k$ cannot be written in the form $\sum_{j\in J }{(q^j-1)}$. Hence  $a_k=0$.
          \end{proof}

\bigskip

Since the subword $W_n$ ends with $0^{\alpha_n}$, we can define $U_n$ as prefix of $W_n$ such that $W_n:=U_n0^{\alpha_n}$, for every $n\geq 1$. Notice that $\left|U_n\right|=q^n-1$.

\begin{lemma}\label{l2} For every $n\geq 1$, we have  $U_{n+1}=U_{n}\widehat{U_{n}}0^{\alpha_{n}}.$
\end{lemma}

\begin{proof} 
By Lemma \ref{concate},  $U_n=(-1)\widehat{W_0}\widehat{W_1}\cdots \widehat{W_{n-1}}$. Consequently: 
$$U_{n+1}=\underbrace{(-1)\widehat{W_0}\widehat{W_1}\cdots \widehat{W_{n-1}}}_{U_n}\widehat{W_n}=U_n\widehat{U_n 0^{\alpha_n}}=U_n\widehat{U_n}0^{\alpha_n}.$$
\end{proof}
\medskip

\begin{rem}\label{rema} Since $q \geq 3$ we have $ \alpha_n\geq |U_n|$  for every $n \geq 1$. Moreover $(\alpha_n)_{n\geq1}$ is a positive and increasing sequence.\end{rem}

%.\\  The length of  $U_n$ is  $$|U_n|=q^n-1$$ and since
%$$\alpha_n=(q^{n+1}-1)-2(q^n-1)$$ we get that $$ \alpha_n\geq |U_n|$$ when $q\geq 3$ for every $n \geq 1$.\\
%2. Moreover, $(\alpha_n)_{n\geq1}$ is a positive and increasing sequence.

\begin{lemma}\label{alphab} For every $n\geq 1$, let $A_n:=\{U_n, \widehat{U_n}, 0^{\alpha_i}, \, i \geq n \}$. Then $\p \in A_n^{\mathbb N}$.
\end{lemma}

\begin{proof} Let $n\geq 1$. By definition of $W_n$ and $U_n$, the infinite word $\p$ can be factorized as:
$$\p=\underbrace{1 W_0 W_1\cdots W_{n-1}}_{V_n}W_n W_{n+1}\cdots.$$

By Lemma \ref{concate},  since $U_n=(-1)\widehat{W_0}\widehat{W_1}\cdots \widehat{W_{n-1}}$ 
then the prefix $V_n=\widehat{U_n}$.

Also, $W_{n+1}=U_n\widehat{U_n}0^{\alpha_n}0^{\alpha_{n+1}}$, $W_{n+2}=U_n\widehat{U_n}0^{\alpha_n}\widehat{U_n}U_n 0^{\alpha_n+\alpha_{n+1}+\alpha_{n+2}}$. Keeping on this procedure, $W_n$ can be written as a concatenation of $U_n$, $\widehat{U_n}$ and $0^{\alpha_i}$, $i\geq n$. More precisely,  $\p$ can be written  in the following manner:
$$\p= \widehat{U_n} \underbrace{U_n0^{\alpha_n}}_{W_n}\underbrace{U_n\widehat{U_n}0^{\alpha_n+\alpha_{n+1}}}_{W_{n+1}}\underbrace{U_n\widehat{U_n}0^{\alpha_n}\widehat{U_n}U_n 0^{\alpha_n+\alpha_{n+1}+\alpha_{n+2}} }_{W_{n+2}}              \cdots.$$\\
\end{proof}

\medskip
\begin{proof}[Proof of Proposition \ref{upper2}]

Let $m\in \N$.  Then there exists a unique positive integer $n$, such that:  $$q^{n-1}-1\leq m < q^n-1.$$

By Lemma \ref{alphab} and  the Remark \ref{rema},  between the words $U_n$ and $\widehat{U_n}$ (when they do not  occur consecutively), there are only blocks of zeros of length greater than $\alpha_n\geq |U_n|=q^n-1$ and thus greater than  $m$.
Hence, all distinct factors of length $m$ appear in the following words:
$U_n\widehat{U_n}$, $\widehat{U_n}U_n$,
$0^{\alpha_n}U_n$, $0^{\alpha_n}\widehat{U_n}$, $U_n0^{\alpha_n}$ and $\widehat{U_n}0^{\alpha_n}$. \\

%$$\widehat{U_n} U_n0^{\alpha_n}U_n\widehat{U_n}0^{\alpha_n}.$$

In $U_n\widehat{U_n}$   we may find at most $|U_n\widehat{U_n}|-m+1=2|U_n|-m+1$  factors at length $m$.  
In $\widehat{U_n}U_n$ we may find at most $m-1$ new different factors of length $m$.  More precisely, they form the set $i(\widehat{U_n},\varepsilon, U_n)^+$.

In  $0^{\alpha_n}U_n$ (respectively $0^{\alpha_n}\widehat{U_n}$, $U_n0^{\alpha_n}$, $\widehat{U_n}0^{\alpha_n}$)
 we may find at most $m$ (respectively $m-1$) new different factors (they belong to $i(0^{\alpha_n}, \varepsilon, U_n)^+ \cup \{0^m\}$, respectively $i(0^{\alpha_n},\varepsilon, \widehat{U_n})^+$, $i(U_n,\varepsilon, 0^{\alpha_n})^+$ and $i(\widehat{U_n}, \varepsilon, 0^{\alpha_n})^+)$.\\

Consequently, the number of such subwords is at most $2|U_n|+4m-3$. Since $U_n =q^n-1 = q(q^{n-1}-1)+q-1 \leq qm+q-1$ we obtain that:
$$p(\p,m) \leq 2(qm+q-1)+4m-3 \leq (2q+4)m+ 2q-3.$$
\end{proof}

\begin{rem}\label{notperiodiq} It is not difficult to prove that $\p$  is not eventually periodic.  Indeed, recall that $\p=W_0W_1W_2\cdots$.  Using  Theorem \ref{concate} and the Remark \ref{rema}, 
\begin{equation}\label{notperiodic}\p=A_10^{l_1}A_20^{l_2}\cdots A_i 0^{l_i}\cdots,\end{equation}
 where $A_i$, $i\geq 1$, are finite words  such that $A_i\neq 0^{|A_i|}$ and $(l_i)_{i\geq 1}$ is a strictly increasing sequence.

\end{rem}

\begin{rem} This part concerns the case where  $q\geq 3$.  If the characteristic of the field is $2$, that is, if $q=2^n$, where $n\geq 2$, then, in the proof we have that $-1=1$, but the structure of $\p$ remain the same. We will  have certainly a lower complexity, but $\p$ is still on the form (\ref{notperiodic}),   and thus  $p(\p,m)\leq (2q+4)m+ 2q-3$.
\end{rem}

\bigskip

\section {Closure properties of two classes of Laurent series}\label{closure}
It is natural to classify  Laurent series in function of their complexity. 
In this section we  study some closure  properties  for   the following classes:

$$\mathcal{P}=   \{f \in \F_q((T^{-1})),  \text{there exists } K \text{ such that  }  p(f,m)= O(m^K)\}$$\\
and, more generally,
$$\mathcal{Z}  =  \{f \in \F_q((T^{-1})),  \text{ such that } \;h(f)=0\}.$$

Clearly, $\mathcal{P} \subset \mathcal{Z} $. We recall that $h$ is the topological entropy defined in Section 2.

We have already seen, in Theorem \ref{algebraic}, that the algebraic Laurent series belong to $\mathcal P$ and $\mathcal Z$. Also, by Theorem \ref{case2} and \ref{caseq}, $\frac{1}{\Pi_q}$ belongs to $\mathcal P$. Hence, $\mathcal P$, and more generally $\mathcal Z$, seem to be two important objects of interest for this classification.

The main result we will prove in this section is Theorem \ref{vector space}.

%\begin{theoreme}\label{vector space} $\mathcal{P}$  and  $\mathcal{Z}$ are  vector spaces over $\F_q(T)$. \end{theoreme}

In the second part, we will prove the stability of $\mathcal P$ and $\mathcal Z$ under Hadamard product, formal derivative and Cartier operator.

\subsection{Proof of  Theorem \ref{vector space}}

The proof of Theorem  \ref{vector space} is a straightforward consequence of  Propositions \ref{addition} and \ref{ration} below.

\begin{prop}\label{addition} Let $f$ and $g$ be two Laurent  series belonging to  $\F_q((T^{-1}))$. Then, for every integer $m \geq 1$, we have:
$$ \frac{p(f,m)}{p(g,m)}\leq p(f+g, m) \leq p(f,m) p(g,m).$$
\end{prop}

\begin{proof}

Let $f(T):=\sum_{i\geq -i_1}{a_iT^{-i}} \text{ and } g(T):=\sum_{i\geq -i_2}{b_iT^{-i}} \,,\; i_1,i_2 \in \N$.  \\

 By definition   of the complexity of  Laurent series (see Section (\ref{defini})), for every $m\in \N$: $$ p(f(T)+g(T),m)=p(\sum_{i\geq 0}{c_iT^{-i}},m),$$
 where $c_i:=(a_i+b_i) \in \F_q$.
Thus we may suppose that $$f(T):=\sum_{i\geq 0}{a_iT^{-i}}\text{ and } g(T):=\sum_{i\geq 0}{b_iT^{-i}}.$$
We denote by $\a:=(a_i)_{i\geq 0}$, $\textbf {b}:=(b_i)_{i\geq 0}$ and $\textbf {c}:=(c_i)_{i\geq 0}$.

For the sake of simplicity, throughout this part, we  set  $x(m):=p(f,m)$ and  $y(m):=p(g,m)$.
Let $\mathcal{L}_{f,m}:=\{U_1,U_2,\ldots, U_{x(m)}\}$ (resp. $\mathcal{L}_{g,m}:=\{V_1,V_2,\ldots, V_{y(m)}\}$) be the set of different factors of length $m$  of the sequence of coefficients of $f$ (respectively of $g$). As the sequence of coefficients  of the Laurent series $f+g$ is obtained by the termwise addition of the  sequence of coefficients of  $f$ and the sequence of coefficients of $g$, we deduce that:
$$\mathcal{L}_{f+g,m} \subseteq  \{U_i+V_j, \; 1\leq i \leq x(m), \; 1\leq j\leq y(m)\}$$
where  $\mathcal{L}_{f+g,m}$ is the set of all distinct factors of length $m$ occurring in $\bf c$,  and  where the sum of two words with the same length $A=a_1\cdots a_m$ and $B=b_1\cdots b_m$ is defined as $$A+B= (a_1+b_1)  \cdots (a_m+b_m) $$ (each sum being considered over $\F_q $). 
Consequently,  $p(f+g,m)\leq p(f,m)p(g,m)$.

\bigskip

We shall prove now the first  inequality using   Dirichlet's principle. 

Notice that if $x(m)< y(m)$ the  inequality is obvious. \\

Assume now that   $x(m) \geq y(m)$. Remark that if we extract $x(m)$ subwords  of length $m$ from  $\b$,   there is at least one word  which appears at least $\left\lceil \frac{x(m)}{y(m)}\right\rceil$ times. \\

For every fixed $m$, there exist exactly $x(m)$ different factors of  $\a$.  The subwords of $\bf c$ will be obtained adding factors of length $m$ of  $\a$ with factors of length $m$ of $\b$. 
 
Consider all distinct factors of length $m$ of $\a$:  $U_1, U_2,\ldots, U_{x(m)}$, that occur in positions $i_1,i_2,\ldots, i_{x_m}$. Looking in the same positions in $\bf b$, we have $x(m)$ factors of length $m$ belonging to $ \mathcal{L}_{g,m}$. Since $x(m) \geq y(m)$, by the previous remark, there is one word $W$ which occur at least  $\left\lceil \frac{x(m)}{y(m)}\right\rceil$ times in $\b$.

Since we have $U_i+W \neq U_j+W$ if $U_i\neq U_j$,  the conclusion follows immediately.
\end{proof}

\begin{rem}\label{autre} In fact, the first inequality may also be easily obtained from the second one, but we chose here to give  a more intuitive proof. Indeed, if we denote $f:=h_1+h_2$, $g:=-h_2$, where $h_1,h_2 \in \F_q((T^{-1}))$, the first relation follows immediately, since $p(h_2,m)=p(-h_2,m)$, for any $m\in \N$.
\end{rem}

%For every $m$ fixed, there exist $x(m)$ different factors of  $\a=a_0a_1a_2\cdots$.  The subwords of $\bf c$ will be obtained adding factors of length $m$ of  $\a$ with factors of length $m$ of $\b$. 

  %We have exactly $x(m)$ distinct words of $\a$ and, by the previous remark, there are  at least   $\left\lceil \frac{x(m)}{y(m)}\right\rceil$ distinct words that will be added with the same word. Hence, we obtain at least $\left\lceil \frac{x(m)}{y(m)}\right\rceil$  distinct factors of length $m$ of $\bf c$. This ends the proof. 

\begin{rem}\label{addition_poly} If $f\in \F_q((T^{-1}))$ and $a \in \F_q[T]$ then, obviously,  there exists a constant $C$ (depending on the degree of the polynomial $a$) such that, for any $m\in \N$, $$p(f+a, m) \leq p(f,m)+C.$$
\end{rem}
\medskip

\begin{rem}\label{saturate_eq}
Related to Proposition \ref{addition}, one can naturally  ask if it is  possible to saturate the inequalities  in Proposition \ref{addition}. By Remark \ref{autre}, it suffices to show that this is possible  for one inequality. In the sequel, we construct two explicit examples of Laurent series of linear complexity such that their sum has  quadratic complexity. 

\medskip

 Let $\alpha$ and $\beta$ be two irrational numbers such that $1$, $\alpha$ and $\beta$ are  linearly independent over $\mathbb Q$.
For any  $i\in \{\alpha, \beta\}$ we consider the following rotations:
 $$ R_{i}: \mathbb T^1 \rightarrow \mathbb T^1 \hspace{1cm} x\rightarrow \{x+i \},$$
 %$$ R_{\alpha}: \mathbb T^1 \rightarrow \mathbb T^1 \hspace{1cm} x\rightarrow \{x+\alpha \},$$
 %$$ R_{\beta}: \mathbb T^1 \rightarrow \mathbb T^1 \hspace{1cm} x\rightarrow \{x+\beta \},$$
 where $\mathbb T^1$ is the circle ${\mathbb R}/{\mathbb Z}$, identified to the interval $[0,1 )$.
 
 We may partition $\mathbb T^1$ in two intervals $I_{i}^0$  and $I_{i}^1$, delimited by $0$ and $1-i$. We denote by $\nu_{i}$ the coding function:
 
 \begin{displaymath}
\nu_{i}(x) = \left\{ \begin{array}{ll}
0 & \text{ if } x \in I_{i}^0;\\       
1 & \text{ if  } x \in I_{i}^1.
\end{array} \right.
\end{displaymath}

 %$$\begin{array}{lcc} \nu_{i}(x)&=&0 \text{ if } x \in I_0;\\\nu_{i}(x)&=&1 \text{ if  } x \in I_1.\end{array}$$
    
      We define $\a:=(a_n)_{n\geq 0}$ such that, for any ${n\geq 0}$, $$a_n=\nu_{\alpha}(R_{\alpha}^n(0))=\nu_{\alpha}(\{n\alpha \})$$  and $\b:=(a_n)_{n\geq 0}$ such that, for any ${n\geq 0}$, $$b_n=\nu_{\beta}(R_{\beta}^n(0))=\nu_{\beta}(\{n \beta \}).$$
      
 \medskip     
  Let us consider $f(T)=\sum_{n\geq 0}{a_nT^{-n}}$  and $g(T)=\sum_{n\geq 0}{b_nT^{-n}}$ be two elements of $\F_3((T^{-1}))$.    We will prove that, for any $m\in \N$, we have:
 \begin{equation}\label{saturate}p(f+g,m)=p(f,m)p(g,m).\end{equation}
 We thus provide  an example of two infinite words whose sum has a maximal complexity, in view of Proposition \ref{addition}.
\medskip

 A sequence of form $(\nu(R_{\alpha}^n(x)))_{n\geq 0}$ is a particular case of rotation sequences. It is not difficult to see that the complexity of the sequence $\a$  satisfies $p(\a, m)=m+1$ for any $m\in \N$ and hence $\a$ is Sturmian. For a complete proof, the reader may consult the monograph \cite{pytheas}, but also the original paper of Morse and Hedlund \cite{Morse_Hedlund}, where they  prove  that every Sturmian sequence is a rotation sequence.

\medskip

Let $m\in \N$. Let $$\mathcal L_{\a, m}:= \{U_1, U_2, \ldots, U_{m+1}\}$$ and respectively $$\mathcal L_{\b, m}:= \{V_1, V_2, \ldots, V_{m+1}\}$$ be the set of distinct factors of length $m$ that occur in $\a$, respectively in $\b$.

 In order to prove the relation (\ref{saturate}), we show  that \begin{equation}\label{cartesian}\mathcal L_{\a+\b, m} = \{ U_i+V_j, 1 \leq i, j \leq m+1\}.\end{equation}

Let $I:=[0,1)$. It is well-known (see for example Proposition 6.1.7 in \cite{pytheas}) that, using the definition of the sequence $\a$ (respectively of $\b$),  we can split $I$ in $m+1$ intervals of positive length $J_1,J_2,\ldots, J_{m+1}$ (respectively  $L_1,L_2,\ldots, L_{m+1}$)  corresponding to $U_1, U_2, \ldots, U_{m+1}$ (respectively $V_1, V_2,\ldots, V_{m+1}$) such that:
%\begin{equation}\nonumber J_k=\{\{n\alpha\} \text{ such that } U_k=a_n a_{n+1}\cdots a_{n+m-1}\} \end{equation}
%\begin{equation}\nonumber (\text{respectively } L_k=\{\{n\beta\} \text{ such that } V_k=b_n b_{n+1}\cdots b_{n+m-1}\}). \end{equation}
\begin{equation}\nonumber\{n\alpha\}\in J_k \text{ if and only if }  a_{n}a_{n+1}\cdots a_{n+m-1}=U_k \end{equation}
\begin{equation}\nonumber (\text{respectively }\{n\beta\}\in L_k \text{ if and only if }  b_{n}b_{n+1}\cdots b_{n+m-1}=V_k.)\end{equation}

In other words, $\{n\alpha\}\in J_k$ (resp. $\{n\beta\}\in L_k$) if and only if the factor $U_k$ (resp.$V_k$)  occurs in  $\a$ (resp. $\b$)   at the position $n$.
\medskip

Now we use the well-known Kronecker's theorem which asserts that the sequence of fractional parts $(\{n \alpha\},\{n \beta\})_{n\geq 0}$ is dense in the square $[0,1)^2$ since by assumption $1$, $\alpha$ and $\beta$ are linearly independent over $\mathbb Q$.

 In particular, this implies that, for any pair $(i,j) \in \{0,1,\ldots, m+1\}^2$, there exists a positive integer $n$ such that $(\{n\alpha\}, \{n \beta\} )  \in J_i \times L_k$. This is equivalently to say that, for any pair of factors $(U_i, V_j)\in \mathcal L_{\a, m} \times \mathcal L_{\b, m}$, there exists $n$ such that $U_i=a_{n}a_{n+1}\cdots a_{n+m-1}$ and $V_k=b_{n}b_{n+1}\cdots b_{n+m-1}$. This proves Equality (\ref{cartesian}) and more precisely, since we are in characteristic $3$, we have the following equality:
 $$\mathrm{Card}\, \mathcal L_{\a+\b, m}=\mathrm{Card}\, \mathcal L_{\a, m}  \cdot \mathrm{Card}\, \mathcal L_{\b, m} =(m+1)^2.$$
\end{rem}

\vspace{10mm}

We point out the following consequence of Proposition \ref{addition}. 
\begin{coro}\label{majoration2}
 Let $f_1, f_2,\ldots, f_l \in  \F_q((T^{-1}))$. Then  for every $m\in \N$ and for every integer $i\in [1;l]$ we have the following:
 $$\frac{p(f_i,m)}{\prod_{j\neq i, 1\leq j\leq l}{p(f_j,m)}}\leq p(f_1+f_2+\cdots+f_l,m) \leq \prod_{1\leq j \leq l} {p(f_j,m)}.$$
  \end{coro}

Notice that these inequalities can be saturated, just generalizing the construction above (choose $l$ Sturmian sequences of irrational slopes $\alpha_1, \alpha_2, \ldots, \alpha_l$, such that  $1, \alpha_1, \alpha_2, \ldots, \alpha_l$ are linearly independent over $\mathbb Q$).

%\begin{rem} According to  lemma \ref{addition}, the addition with a rational series does not change the order of complexity.  \end{rem}

\vspace{15mm}

We shall prove next that the sets $\mathcal{P}$ and $\mathcal{Z}$ are closed under multiplication by rationals. Let us begin with a particular case, that is the multiplication by a polynomial.
\begin{prop}  \label{poly} Let $b(T)\in \F_q[T]$ and $f(T)\in \F_q((T^{-1}))$. Then there is a positive constant $M$ (depending only on $b(T)$), such that for all $m\in \N$:
 $$p(bf,m)\leq M \;p(f,m).$$
\end{prop}

\begin{proof}  Let $$b(T):=b_0T^r+b_1T^{r-1}+\cdots+b_r \in \F_q[T]$$ and $$f(T):=\sum_{i\geq -i_0}{a_iT^{-i}}\in \F_q((T^{-1})), \, i_0\in \N.$$ 

Then \begin{equation}\label{poly1}\begin{split} b(T)f(T)&=b(T) \left(\sum _{i= -i_0}^{-1}{a_iT^{-i}}+\sum_{i\geq 0}{a_iT^{-i}}\right)\\ & =b(T)\left(\sum _{i= -i_0}^{-1}{a_iT^{-i}}\right)+b(T)\left (\sum_{i\geq 0}{a_iT^{-i}}\right). \end{split}\end{equation}

Now, the product
 \begin{equation}\label{poly2} \begin{split}b(T)(\sum_{i\geq 0}{a_iT^{-i}})&=T^r (b_0+b_1T^{-1}+b_2 T^{-2}+\cdots+b_rT^{-r})( \sum_{i\geq 0}{a_iT^{-i}})\\&:=T^r(\sum_{j\geq 0}{c_jT^{-j}})\end{split}\end{equation}
 
where the sequence   $\textbf{c}:=(c_j)_{j\geq 0}$ is defined as follows:
\[c_j=
\begin{cases} b_0a_j+b_1a_{j-1}+\cdots+b_ja_0  \text { if }  j<r \\
 b_0a_j+b_1a_{j-1}+\cdots+b_ra_{j-r}  \text{ if } j\geq r. \end{cases}\]

%First of all, notice that the multiplication of $f(T)$ by a monomial $s(T)=T^{s_0}$, does not change the order of magnitude of the complexity  of $f$. More precisely, $p(s(T)f(T),m)=p(f(T),m)+ S$, where $S$ is a parameter that depends only on $s_0$. 

According to definition of complexity (see Section \ref{defini}) and to relations (\ref{poly1}) and (\ref{poly2}), for every $m\in \N$, we have $$p\left(b(t)f(T),m\right)=p\left(b(T)(\sum_{i\geq 0}{a_iT^{-i}}),m\right)=p\left((\sum_{j\geq r}{c_jT^{-j}}),m\right).$$

Our aim is to count the number of words of form $c_jc_{j+1}\cdots c_{j+m-1}$, when $j\geq r$.
By definition of  $\textbf{c}$, we notice that for  $j\geq r$ these words depend only on  $a_{j-r}a_{j-r+1}\cdots a_{j+m-1}$ and of $b_0,b_1,\cdots, b_r$, which are fixed. The number of words $a_{j-r}a_{j-r+1}\cdots a_{j+m-1}$ is exactly  $p(f,m+r)$.  By Lemma \ref{add} we obtain: 
$$p(f,m+r) < p(f,r)p(f, m)=M p(f,m),$$
where $M=p(f,r)$. More precisely, we may bound up $M$ by $q^r$, since this is  the number of  all possible  words of length $r$ over an alphabet of $q$ letters. 
 \end{proof}

\begin{prop}\label{ration} Let  $r(T)\in \F_q(T)$  and $f(T)=\sum_{n\geq -n_0}{a_n T^{-n}}\in \F_q((T^{-1}))$. Then for every $m\in \mathbb{N}$, there is a positive constant $M$, depending only on $r$ and $n_0$,  such that:
$$p(rf,m)\leq M p(f,m).$$
\end{prop}

\begin{proof}

Let $f(T):=\sum_{i\geq -i_0}{a_iT^{-i}}\in \F_q((T^{-1}))$, $i_0\in \N$ and  $m\in \N$. By Proposition \ref{addition}, we have:
$$p(r(T)f(T),m) \leq  p\left(r(T)(\sum_{i= -i_0}^{-1}{a_iT^{-i}}),m\right) \cdot p\left(r(T)(\sum_{i\geq 0}{a_iT^{-i}}),m\right).$$   
 Proposition \ref{poly} implies that $$p\left(r(T)(\sum_{i= -i_0}^{-1}{a_iT^{-i}}),m\right) \leq R$$ where $R$ does not depend on $m$.  Thus, we may assume that $f(T)=\sum_{i\geq  0}{a_iT^{-i}}$.\\

The proof of  Proposition \ref{ration} is decomposed into five steps.\\

\textbf{Step 1.} Since $r(T) \in \F_q(T)$, the sequence of coefficients of $r$  is eventually periodic. Thus, there exist two positive integers $S$ and $L$ and two polynomials $p_1\in \F_q[T]$ (with degree equal to $S-1$) et $p_2\in \F_q[T]$ (with degree  equal to $L-1$) such that  $r$ may be written as follows:
$$r(T)=\frac{P(T)}{Q(T)}=\frac{p_1(T)}{T^{S-1}}+\frac{p_2(T)}{T^{S+L-1}} (1+T^{-L}+T^{-2L}+\cdots).$$
    
  Hence  
\begin{equation}\label{eq5}\begin{split}r(T)f(T)&=\underbrace{\frac{1}{T^{S-1}}p_1(T)f(T)}_{g(T)}+\underbrace{p_2(T)\frac{1}{T^{S+L-1}}f(T)(1+T^{-L}+T^{-2L}\cdots)}_{h(T)}\\&:=\sum_{n\geq 0 }{f_nT^{-n}}.\end{split}\end{equation}
 
 Let us denote by $\textbf{d}=(d(n))_{n\geq 0}$ the sequence of coefficients of $g(T)$ and by $\textbf{e}=(e_n)_{n\geq 0}$ the sequence of coefficients of $h(T)$. Clearly $\textbf f:=(f_n)_{n\geq 0}$ is such that $f_n=d_n+e_n$, for every $n\in \N$.\\

 Fix $m\in \N$. Our aim is to bound from above $p(\textbf{f},m)$. First,  assume that $m$ is a multiple of $L$. More precisely, we set $m=kL$, where $k\in \N$. \\

 In order to bound the  complexity of $\textbf{f}$, we will consider separately  the sequences $\textbf e$ and $\textbf d$.
 
\bigskip

\textbf{Step 2.} We study now the sequence $\textbf e$, defined in (\ref{eq5}). 

In order to describe the sequence $\textbf e$, we shall study first the  product $$f(T)(1+T^{-L}+T^{-2L}+\cdots)=(\sum_{i\geq 0}{a_iT^{-i}})(1+T^{-L}+T^{-2L}+\cdots):=\sum_{j\geq 0}{c_j T^{-j}}.$$

Expanding this product, it is not difficult to see that:
 $c_l=a_l$ if $l<L$ and $c_{k L+l}=a_l+a_{l+L}+\cdots + a_{k L+l}$, for $k\geq 1$ and $0\leq l \leq L-1$.\\

By definition of $c_n$,  $n\in \mathbb{N}$, we can easily obtain:
$$c_{n+L}-c_n=a_{n+L}.$$
Consequently, for all  $s\in \mathbb{N}$:
\begin{equation}\label{eq2}c_{n+s L}-c_n=a_{n+sL}+a_{n+(s-1)L}+\cdots+a_{n+L}.\end{equation}

Our goal  is now to study the subwords of $\textbf c$  with length $m=k L$.

 Let $j\geq 0$ and let $c_j c_{j+1}c_{j+2}\cdots c_{j+k L-1}$ be a finite factor of length $m=kL$. Using identity (\ref{eq2}), we may split the factor above in $k$ words of length $L$ as follows:  
 \begin{equation}  \nonumber \begin{split} c_jc_{j+1}c_{j+2}\cdots c_{j+k L-1}=& \underbrace{c_j c_{j+1}\cdots c_{j+L-1}}_{D_1}\underbrace{c_{j+L}c_{j+L+1}\cdots c_{j+2L-1}}_{D_2}\cdots \\ & ...\underbrace{c_{j+(k-1)L}c_{j+(k-1)L+1}\cdots c_{j+kL-1}}_{D_k}\end{split}\end{equation}
 where the words  $D_i$, $2\leq i \leq k$ depend only on $D_1$ and $\a$. More precisely, we have:
 \begin{align*}
D_2 &=(c_j+a_{j+L})(c_{j+1}+a_{j+L+1})\cdots (c_{j+L-1}+a_{j+2L-1})\\
%D_3 & =(c_j+a_{j+L}+a_{j+2L})(c_{j+1}+a_{j+L+1}+a_{j+2L+1})\cdots\\   & \hspace{0.5cm} (c_{j+L-1}+a_{j+2L-1}+a_{j+3L-1})\\
\vdots & \\
D_k&=(c_j+a_{j+L}+\cdots+a_{j+(k-1)L})(c_{j+1}+a_{j+L+1}+\cdots+a_{j+(k-1)L+1})\cdots\\ 
  & \hspace{0.5cm}(c_{j+L-1}+a_{j+2L-1}+\cdots+a_{j+kL-1}).\end{align*}

 Consequently,  the word $ c_j c_{j+1}c_{j+2}\cdots c_{j+m-1}$ depends only on $D_1$, which is a factor of   length $L$, determined by $r(T)$, and on the word $a_{j+L}\cdots a_{j+k L-1}$, factor of length $k L-L=m-L$ of $\a$.\\

 Now, let us return to the sequence $\textbf e$. We recall that \begin{equation}\label{eq4}\sum_{n\geq 0}{e_n T^{-n}}=\frac{p_2(T)}{T^{S+L-1}}\sum_{j \geq 0}{c_j T^{-j}}.\end{equation}
Using a similar argument as in the proof of Proposition \ref{poly} and using the identity (\ref{eq4}), a factor of the form $e_j e_{j+1}\cdots e_{j+m-1}$, $j\in \N$, depends only on the coefficients of $p_2$, which are fixed, and on   $c_{j-L+1}\cdots c_{j-1}c_j\cdots c_{j+m-1}$. Hence, the number of distinct factors of the form $e_j e_{j+1}\cdots e_{j+m-1}$ depends   only on the number of distinct factors of the  form $a_{j+1}a_{j+2}\cdots a_{j+(k-1)L}$ and  on the number of factors of length $L$ that occur in $\textbf c$.\\

\textbf{ Step 3.} We describe now the sequence $\textbf d$, defined in (\ref{eq5}).

Doing the same proof as for Proposition \ref{poly}, we obtain that the number of words $d_j \cdots d_{j+m-1}$, when $j\in \N$, depends only on the coefficients of $p_1$, which are fixed, and on the number of distinct factors $a_{j-S+1}\cdots a_{j}\cdots a_{j+m-1}$.\\

\textbf{Step 4.} We now give an upper bound for the  complexity of $\textbf f$, when $m$ is  a multiple of $L$.

According to steps 2 and 3, the number of distinct factors of the form $f_j f_{j+1}\cdots f_{j+m-1}$, $j\in \N$, depends on the number of distinct factors of form  $a_{j-S+1}a_{j+2}\cdots a_{j+m-1}$ and on the number of factors of length $L$ that occur in $\textbf c$.
 
Consequently, 
 $$p(r f, m)\leq  p(f,m+S-1)q^L, $$
and by Lemma \ref{add} 
$$p(f,m+S-1)\leq p(f,m)p(f,S-1) \leq q^{S-1} p(f,m).$$

Finally, $$p(r f,m)\leq q^{L+S-1}p(f,m).$$

\textbf{Step 5.} We now give an upper bound for the complexity of $\textbf f$, when $m$ is not a multiple of $L$.

In this case,  let us suppose that $m=kL+l$, $1\leq l \leq L-1$. Using Lemma \ref{add} and according to Step 4: \begin{eqnarray*} p(rf, m)&= &p(rf, kL+l)\leq p(rf,kL)p(rf,l)\leq p(rf,kL) p(rf,L-1)\\ & \leq & q^{L-1} p(rf,kL) \leq q^{S+2L-2}p(f, m).\end{eqnarray*}
 \end{proof}

\subsubsection{A  criterion for linear independence of Laurent series}

As a consequence of Theorem \ref{vector space}, we  give a  criterion of linear independence over $\F_q(T)$ for two Laurent series  in function of their complexity. 
\begin{prop}\label{indep1}
Let $f, g\in \F_q((T^{-1}))$ be two irrational Laurent series such that:
$$\lim_{m\rightarrow \infty} \frac{p(f,m)}{p(g,m)}=\infty.$$ Then $f$ and $g$  are linearly independent over the field  $\F_q(T)$.
\end{prop}

\begin{proof} We argue by contradiction. Assume there exist polynomials $A(T)$, $B(T)$, $C(T)$ over $ \F_q$, not all zeros, such that:
$$A(T)f(T)+B(T)g(T)+C(T)=0.$$
Next use the fact that addition with a rational function and multiplication by  a rational function  do not increase the  asymptotic order  of complexity. Indeed, since $A(T)\neq 0$  because $g(T) \notin \F_q(T)$,  we would have $$f(T)+\frac{C(T)}{A(T)}=-\frac{B(T)}{A(T)}g(T).$$ 
However,  Propositions \ref{addition} and \ref{ration} would imply that the complexity of the left-hand side of this inequality is asymptotically larger than the one of the right-hand side. 
\end{proof}

\medskip

Let us now  give an example of two Laurent  series linearly independent over  $\F_q(T)$. Their sequences of coefficients are generated by non-uniform morphisms and  we study their subword complexity in function of the order of growth of letters, using a classical result  of Pansiot \cite{Pansiot}. Notice that,  the  following  sequences are non-automatic and hence, the associated Laurent series are transcendental over $\F_q(T)$.

\begin{ex} Consider the infinite word $\a=000100010001110\cdots$;  $\a=(a_n)_{n\geq 0}=\sigma^{\infty}(0)$ where $\sigma(0)=0001$ and $\sigma(1)=11$. If  we look to the   order of growth of $0$ and $1$ we have that $\left|\sigma^n(0)\right|=3^n+5\cdot 2^{n-2}$ and $\left|\sigma^n(1)\right|=2^n$. Hence, the morphism  $\sigma$ is exponentially diverging (see the Section (\ref{morphisms})). Consequently, by Pansiot's theorem mentioned above, $p(\a,m)=\Theta(m \log m)$.

\medskip

Consider next $\b=010110101111010\cdots$; $\b=(b_n)_{n\geq 0}=\phi^{\infty}(0)$, where $\phi(0)=0101$ and $\phi(1)=11$.  It is not difficult to see that $\phi$ is polynomially diverging (see Section (\ref{morphisms}))  since $\left|\phi^n(0)\right|=(n+1)2^n$ and $\left|\phi(1)^n\right|=2^n$. By Pansiot's theorem, $p(\b, m)=\Theta(m \log \log m)$.

\medskip

Now  we consider the formal series whose coefficients are the sequences generated by the morphisms above:
$$f(T)=\sum_{n\geq 0}{a_n T^{-n}}=\frac{1}{T^3}+\frac{1}{T^7}+\frac{1}{T^{11}}+\frac{1}{T^{12}}+\cdots \in \F_q[[T^{-1}]]$$ and 
$$g(T)=\sum_{n\geq 0}{b_nT^{-n}}=\frac{1}{T^1}+\frac{1}{T^3}+\frac{1}{T^4}+\frac{1}{T^6}+\cdots  \in \F_q[[T^{-1}]].$$
Since $\lim_{m\rightarrow \infty}p(f,m)/p(g,m)=+\infty$, Proposition \ref{indep1} implies that $f$ and $g$ are linearly independent over $\F_q(T)$.
\end{ex}

%More generally, using the theorem \ref{indep1} and the Pansiot theorem, we may obtain some classes of formal power series associated with different growing morphisms, which are  linear independent over $\F_q(T)$. % if, again, we consider words over finite alphabet of $q$ letters.

\subsection{Other closure properties}
In this section we  prove that the classes $\mathcal{P}$ and $\mathcal{Z}$ are closed  under a number of actions such as:  Hadamard product, formal derivative and Cartier operator. 

\subsubsection{Hadamard product}
 Let $f(T):=\sum_{n\geq -n_1}{a_nT^{-n}}$,  $g(T):=\sum_{n\geq -n_2}{b_nT^{-n}}$ be two Laurent series in $\F_q((T^{-1}))$.
The Hadamard product of $f$ and $g$ is defined as follows:
$$f\odot g=\sum_{n\geq -\min(n_1,n_2)}{a_n b_n T^{-n}}.$$

As in the case of addition of two Laurent series (see Proposition \ref{addition}) one can easily obtain the following.
\begin{prop}\label{hadamard} Let $f$ and $g$ be two Laurent  series belonging to  $\F_q((T^{-1}))$. Then, for every  $m\in \mathbb{N}$, we have:
$$ \frac{p(f,m)}{p(g,m)}\leq p(f\odot g, m) \leq p(f,m) p(g,m).$$
\end{prop}

The proof is similar to the one of Proposition \ref{addition}. The details are left to the reader.

\subsubsection{Formal derivative}

As an easy application of Proposition \ref{hadamard}, we present here the  following result. First, let us recall the definition of the formal derivative.

  \begin{definition}Let  $n_0 \in \N$ and consider the Laurent series: $f(T)=\sum_{n=-n_0}^{+\infty}{a_nT^{-n}} \in \F_q((T^{-1})).$
  The formal derivative of $f$ is defined as follows: $$f'(T)=\sum_{n=-n_0}^{+\infty}{(-n\text{ mod } p)a_n T^{-n+1}} \in \F_q((T^{-1})).$$
\end{definition}

  We prove then the following result.
  
\begin{prop} Let $f(T)\in \F_q((T^{-1}))$ and $k$  be a positive integer. If $f^{(k)}$ is the  derivative of order $k$ of $f$, then there exists a positive constant $M$, such that, for all $m\in \N$, we have:
$$p(f^{(k)},m)\leq M \; p(f,m).$$
\end{prop}

\begin{proof}

The derivative of order $k$ of $f$ is almost the Hadamard product of the series by a rational function. By definition of $p(f,m)$, we may suppose that   $f(T):=\sum_{n\geq 0}{a_nT^{-n}} \in \F_q[[T^{-1}]]$. Then:
$$f^{(k)}(T)=\sum_{n\geq k}{((-n)(-n-1)\cdots (-n-k+1) a_n)  T^{-n-k}}:=T^{-k}\sum_{n\geq k}{b_n a_n T^{-n}},$$
where $b_n:= (-n)(-n-1)\cdots (-n-k+1) \mod p $.
Since $b_{n+p}=b_n$, the sequence $(b_n)_{n\geq 0}$ is periodic of period $p$. Hence, let us denote by $g(T)$ the  series whose coefficients are precisely given by $(b_n)_{n\geq 0}$. Thus there exists a positive constant $M$  such that:
$$p(g,m)\leq M.$$

By Proposition \ref{hadamard},
$$p(f^{(k)},m)\leq p(g,m)p(f,m)\leq M p(f,m),$$
which completes the proof.
\end{proof}

\subsubsection{Cartier's operators}
In the fields of positive characteristic, there is a natural operator, the so-called ``Cartier operator'' that plays an important role in many problems in algebraic geometry  and  arithmetic in positive characteristic \cite{christol_2, christol_1, derksen,sherif}. In particular, if we consider the field of Laurent series with coefficients in $\F_q$, we have the following definition.
\begin{definition}
Let $f(T)=\sum_{i\geq 0}{a_iT^{-i}}\in \F_q[[T^{-1}]]$ and  $0\leq r < q$. The Cartier operator $\Lambda_r$ is  a linear transformation  such that:
$$\Lambda_r( \sum_{i\geq 0}{a_iT^{-i}})= \sum_{i\geq 0}{a_{qi+r}T^{-i}}.$$
\end{definition}

The classes $\mathcal P$ and $\mathcal Z$ are closed under this operator. More precisely, we prove the following result.

\begin{prop} Let $f(T)\in \F_q[[T^{-1}]]$ and $0\leq r < q$. Then there is  $M$ such that, for every $m\in \mathbb N$ we have the following:
$$p(\Lambda_r(f),m) \leq q p(f,m)^q.$$
\end{prop}

\begin{proof} %Let $f(T)=\sum_{n\geq 0}{a_n T^{-n}}$. 
Let $\a:=(a_n)_{n\geq 0}$ be the sequence of coefficients of $f$ and $m\in \N$. In order to compute  $p(\Lambda_r(f),m)$, we have to look at  factors of the form $$a_{qj+r}a_{qj+q+r}\cdots a_{qj+(m-1)q+r},$$ for all $j \in \mathbb N$. But these only depend on  factors of the form $$a_{qj+r}a_{qj+r+1}\cdots a_{qj+(m-1)q+r}.$$ Using  Lemma \ref{add}, we obtain that:
$$p(\Lambda_r(f),m)\leq p(f,  (m-1)q+1) \leq q p(f,m-1)^q  \leq q p(f,m)^q.$$
\end{proof}

\section{Cauchy product of Laurent series}

In the previous section, we  proved that $\mathcal P$ and $\mathcal Z$ are vector space over $\F_q(T)$. This raises naturally the question whether or not these classes form a ring; i.e., are they closed under the usual Cauchy product? 
There are actually some particular cases of Laurent series with low complexity whose product still belongs to $\mathcal P$.
%%b 
In this section we discuss the case of automatic Laurent series. However, we are not able to prove whether $\mathcal P$ or $\mathcal Z$ are or not rings or fields.

\subsection{Products of automatic Laurent series}
A particular case of Laurent series stable by multiplication is the class of $k$-automatic series, $k$ being a positive integer:
$$\text{Aut}_k=\{f(T)=\sum_{n\geq 0}{a_n T^{-n}}\in \F_q((T^{-n})), \, \a=(a_n)_{n\geq 0} \, \text {is }k\text{-automatic} \}.$$  

Since any $k$-automatic sequence has at most a linear complexity, $\text{Aut}_k \subset \mathcal P$. 
A theorem of Allouche and Shallit \cite{Allouche_Shallit2} states that the set $\text{Aut}_k$ is a ring.

 In particular, this implies that, if $f$ and $g$ belong to $ \text{Aut}_k$, then $p(fg, m)= O(m)$. 
 %%b 
Notice also that, in the case where $k$ is a power of $p$, the characteristic of the field $\F_q((T^{-1}))$, the result follows from Christol's theorem.

\begin{rem}However, we do not know whether or not this property is still true if we replace $\text{Aut}_k $ by $\cup_{k\geq 2} \text{Aut}_k$. More precisely, if we consider two Laurent series $f,g \in (\cup_{k\geq 2} \text{Aut}_k)$ we do not know if the product $fg$ is still in $\mathcal {P}$. The next subsection gives a particular example of two Laurent series belonging to $(\cup_{k\geq 2} \text{Aut}_k)$ and such that the product $fg$ is still in $\mathcal P$.
\end{rem}

\bigskip

\subsubsection{Some lacunary automatic Laurent series}
We consider now some particular examples of lacunary  series. More precisely, we focus on the product of series of form: $$f(T)=\sum_{n\geq 0}{T^{-d^n}}\in \F_q((T^{-1})).$$ 
It is not difficult to prove that $p(f, m)=O(m)$. The reader may  refer  to \cite{gheorghiciuc} for more general results concerning the complexity of lacunary  series. The fact that the complexity of $f$ is linear is implied also by the fact that $f \in  \text{Aut}_d$. Notice also that $f$ is transcendental over $\F_q(T)$ if $q$ is not a power of $d$. This is an easy consequence of Christol's theorem and  a theorem of Cobham \cite{Cobham1}.

In this section we will prove the following result.

\begin{theoreme}\label{lacun} Let $d$ and $e$ be two  multiplicatively independent positive integers (that is $\frac{\log d}{\log e}$ is irrational) and let $f(T)=\sum_{n\geq 0}{T^{-d^n}}$ and $g(T)=\sum_{n\geq 0}{T^{-e^n}}$ be two Laurent series in $\F_q((T^{-1}))$.  Then:
$$ p(fg,m)=O(m^4).$$
\end{theoreme}

\begin{rem} The series $f$ and $g$ are linearly independent over $\F_q(T)$. More generally, any two irrational Laurent series, $d$-automatic and respectively $e$-automatic, where $d$ and $e$ are two multiplicatively independent positive integers, are linearly independent over $\F_q(T)$. This follows by a Cobham's theorem.
\end{rem}

Let us denote by $h(T):=f(T)g(T)$. Then $h(T)=\sum_{n\geq 0}{a_n T^{-n}}$ where the sequence $\a=(a_n)_{n\geq 0}$ is defined as follows: \begin{equation}\nonumber a_n:=(\text{the number of pairs }(k,l)\in \N^2 \text{ that verify }n=d^k+e^l) \mod p.\end{equation}

The main clue of the proof is the following consequence of the theory of $S$-unit equations (see \cite{Adamczewski_Bell} for a  proof).
\begin{lemma} Let $d$ and $e$ be two  multiplicatively independent positive integers. There is a finite number of solutions $(k_1,k_2,l_1,l_2)\in \N^4$, $k_1\neq k_2$, $l_1\neq l_2$, that satisfy the equation: $$d^{k_1}+e^{l_1}=d^{k_2}+e^{l_2}.$$
\end{lemma}
Obviously, we have the following  consequence concerning the sequence $\a=(a_n)_{n\geq 0}$:

\begin{coro} \label{sunit} There exists a positive integer $N$ such that, for every $n\geq N$ we have $a_n \in \left\{ 0,1 \right\}$. Moreover, 
$a_n=1$  if and only if there exists one unique pair $(k,l)\in \N^2$ such that $n=d^k+e^l$.
\end{coro}

We prove now  the Theorem \ref{lacun}. For the sake of simplicity, we consider  $d=2$ and $e=3$, but the proof is exactly the same in the general case.

\begin{proof}

Let $\b:=(b_n)_{n\geq 2}$ and $ \textbf{c}:=(c_n)_{n\geq 2}$ be the sequences defined as follows:
 \begin{displaymath}
b_n = \left\{ \begin{array}{ll}
1 & \text{ if there exists a pair $(k,l)\in \N^2$ such that }  n=2^k+3^l, \, 2^k > 3^l;\\
0 & \text{otherwise},
\end{array} \right.
\end{displaymath}

\begin{displaymath}
c_n = \left\{ \begin{array}{ll}
1 & \text{ if there exists a pair $(k,l)\in \N^2$ such that }  n=2^k+3^l, \, 2^k  < 3^l;\\
0 & \text{otherwise}.
\end{array} \right.
\end{displaymath}

Let us denote by $h_1(T):=\sum_{n\geq 2}{b_n T^{-n}}$ and resp. $h_2(T):=\sum_{n\geq 2}{c_n T^{-n}}$ the series associated to $\b$ and $\textbf {c}$.
Using Corollary \ref{sunit}, there exists a polynomial $P\in \F_q[T]$, with degree less than $N$, such that $h$ can be written as follows:   $$h(T)=h_1(T)+h_2(T)+P(T).$$ 
By Remark \ref{addition_poly}, there is  $C\in \mathbb R$ such that, for any $m\in \N$:
$$p(h,m)\leq p(h_1 +h_2, m)+C.$$
In the sequel, we will  show that $p(h_1,m)=p(h_2,m)=O(m^2)$.  Theorem \ref{lacun} will then follow  by  Proposition \ref{addition}.

 We study now the subword complexity of the sequence of coefficients $\b:=(b_n)_{n\geq 2}$. 
The proof is  similar to the proofs of Theorems \ref{case2} and \ref{caseq}. The  complexity of  the sequence $\textbf c$ can be treated in essentially the same way as for $\b$.

\medskip

 \textbf{Step 1.} For all $n\geq 1$, we denote by  $W_n$ the factor of $\b$ that occurs between positions $2^n+1$ and $2^{n+1}$, that is:  
$$W_n:=b_{2^n+3^0}b_{2^n+2}b_{2^n+3^1}\cdots b_{2^{n+1}}.$$ We also set $W_0:=1$.

Observe that $\left|W_n\right|=2^n$. 

With these notations the infinite word $\b$ can be factorized as:
       \begin{equation}\label{conq}\b=\underbrace{1}_{W_0}\underbrace{10}_{W_1}\underbrace{1  0 10 }_{W_2}\underbrace{10100000}_{W_3}\cdots=W_0W_1W_2\cdots.\end{equation}
       
\medskip
       
 \textbf{Step 2.} Let $n \geq 1$ and  $m_n$ be the greatest  integer such that $ 2^n+3^{m_n} \leq 2^{n+1}$. This is equivalently to say that $m_n$ is such that $$2^n+3^{m_n} < 2^{n+1} < 2^{n}+3^{m_n+1}.$$ Notice also that  $m_n= n\lfloor \log_3 2\rfloor$.

  With these notations we have (for $n\geq 5$): $$W_n= 1010^51\cdots 10^{\alpha_i}\cdots 10^{\alpha_{m_n}}10^{\beta_n},$$ where $\alpha_i=2\cdot 3^{i-1}-1$, for $1\leq i \leq m_n$, and $\beta_n= 2^n-3^{m_n} \geq 0$. \\
  
   Let us denote by $U_n$ the prefix of $W_n$ such that $W_n:=U_n0^{\beta_n}$.
   
   Notice that $(m_n)_{n\geq 0}$ is an increasing sequence. Hence $(\alpha_{m_n})_{n\geq 0}$ is increasing. Consequently,  $U_n \prec_p U_{n+1}$ and more generally, $U_n \prec_p W_i$, for every $i\geq n+1$.
   
\medskip
   
\textbf{ Step 3.}  Let $M \in \N$. Our aim is to bound from above the number of distinct factors of length $M$ occurring in $\b$. In order to do this, we will show   that there exists an integer $N$ such that  all these factors occur in $$W_0W_1\cdots W_N $$ or in the set  
%%B 
$$A_{0}:=\{Z \in \mathcal{A}^M; \, Z \text{ is of the form } 0^jP \text{ or } 0^i10^jP, \, P \prec_p U_{N},\, i, j \geq 0, \}.$$
 \medskip
   
  Let $N=\lceil  \log_2(M+1) \rceil+3$. Doing a simple computation   we obtain that $\alpha_{m_N}\geq M$. Notice also that, for any $i \geq N$ we have $$\alpha_{m_i}\geq M.$$ This follows since  $(\alpha_{m_n})_{n\geq 0}$ is an  increasing  sequence.
  
  \medskip
  
  Let $V$ be a factor of length $M$ of $\b$. Suppose that $V$ does not occur in the prefix $W_0 W_1 \cdots W_{N}$. Then, by (\ref{conq}), $V$ must occur in $W_{N}W_{N+1}\cdots$. Hence, $V$ must appear in  some $W_i$, for $i \geq N+1$,
   or in $ \bigcup_{i\geq {N}} i(W_i, \varepsilon , W_{i+1})$.
   
   \medskip
   
   Let us suppose that $V$ occurs in $ \bigcup_{i\geq {N}} i(W_i, \varepsilon , W_{i+1})$.
  Since $W_i$ ends with $0^{\alpha_{m_i}}10^{\beta_i}$, with $\alpha_{m_i}\geq M$, and since $W_{i+1}$ begins with $U_{N}$ and $\left| U_{N}\right| =3^{m_{N}}+1\geq M$, we have that $$\mathcal {A}^M \cap (\bigcup_{i\geq {N}} i(W_i,\varepsilon, W_{i+1})) \subset A_0.$$

 Hence, if $V$ occurs in $ \bigcup_{i\geq {N}} i(W_i, \varepsilon , W_{i+1})$ then $V \in A_0$.
 
 \medskip

 Let us suppose now that $V$ occurs in some $W_i$, for $i \geq N+1$.
  By definition of $W_i$ and $\alpha_i$, for  $i \geq N+1$ and by the fact that  we have:  
  $$W_i= 1010^51\cdots 10^{\alpha_{m_N}}10^{\alpha_{m_N+1}}\cdots10^{\alpha_{m_i}} 10^{\beta_i}=U_{N}0^{\alpha_{m_N+1}}\cdots 10^{\alpha_{m_i}}10^{\beta_i}.$$ 
  
 By assumption, $V$ does not occur in  $W_0 W_1 \cdots W_{N}$; hence $V$ cannot occur in $U_N$ which by definition is a prefix of $W_N$. Consequently, $V$ must be of the form $0^r 1 0^s$, $r,s \geq 0$. Indeed, since $\alpha_{m_N}\geq M$,  all blocks of zeros that follow after $U_{N}$ (and before the last digit $1$ in  $W_i$) are all longer than $M$. But the words of form $0^r 1 0^s$, $r,s \geq 0$ belong also to $A_0$.
  \medskip

Hence, we proved that if $V$ does not occur in the prefix  $W_0 W_1 \cdots W_{N}$, then $V$ belongs to $A_{0}$,  as desired.

\medskip

\textbf{Step 4.} In the previous step we showed that all distinct factors  of length $M$ occur in the prefix $W_0W_1\cdots W_{N}$ or in the set $A_{0}$. 

 Since $$\left|W_0W_1\cdots W_{N} \right|= \sum_{i=0}^{N}{2^i}=2^{N+1}-1$$ and since $N=\lceil  \log_2(M+1) \rceil+3$  we have that:  $$2^{N+1}-1\leq 2^{\log_2(M+1)+5}-1=32M+31,$$ and the number of distinct factors that occur in $W_0W_1\cdots W_{N}$ is less or equal to
$32M+31$.

Also, by an easy  computation, we obtain that the cardinality of the set $A_0$ is $$\mathrm{Card} \,A_{0}= \frac{M^2}{2}+\frac{3M}{2}.$$

Finally, $p(\b, m)= p(h_1,m)=O(m^2)$. In the same manner, one could prove that $p(h_2,m)=O(m^2)$. This achieves the proof of Theorem \ref{lacun}, in view of Proposition \ref{addition}.
\end{proof}  
   
 \subsection{A more difficult case}
 
 Set
 $$\theta(T):= 1+2\sum_{n\geq 1}{T^{-n^2}}\in \F_q((T^{-1})), \, q\geq 3.$$
 The function $\theta(T)$ is related to the classical Jacobi theta function.
  The sequence of coefficients of $\theta(T)$ corresponds to the characteristic sequence of squares and one can easily  prove that:
 $$p(\theta, m)=\Theta(m^2).$$
 In particular this implies the transcendence of $\theta(T)$ over $\F_q(T)$, for any $q\geq 3$. Notice that this implies the transcendence over $\mathbb Q(T)$ of   the same Laurent series but viewed as an element of $\mathbb Q((T^{-1}))$.
 
 Since $\theta(T) \in \mathcal P$, it would be interesting to know whether or not $\theta^2(T)$ belongs also to $\mathcal P$. 
Notice that  $$\theta^2(T)=\sum_{n\geq 1}r_2(n)T^{-n}$$ where $r_2(n)$ is the number of representations of $n$ as sum of two squares of integers mod $p$. 

In the rich bibliography concerning Jacobi theta function (see for instance \cite{duverney, hardy}), there is the following  well-known  formula:
 
$$ r_2(n)=4( d_1(n)-d_3(n)) \text{ mod } p $$ where $d_i(n)$ denotes the number of divisors of $n$ congruent to $i$ modulo $4$, for each $i\in \{1,3\}$.

In particular, by an easy consequence of Fermat's $2$-squares theorem  we can deduce that $r_2(n)=0$ if $n$ is a prime of the form $4k+3$ and $r_2(n)=8 \mod p$ if $n$ is a prime of the form $4k+1$.
  
More generally, if 
  $$n=2^{\gamma}p_1^{\alpha_1}p_2^{\alpha_2}\cdots p_k^{\alpha_k}q_1^{\beta_1}q_2^{\beta_2}\cdots q_l^{\beta_l},$$ where  $p_i\equiv 1\;[4]$ et $q_j\equiv 3 \;[4]$  then
   \begin{displaymath}
r_2(n) = \left\{ \begin{array}{ll}
0 \text{ if there exists an odd $\beta_j$ in the decomposition of $n$};\\
4(\alpha_1
+1)(\alpha_2+1)\cdots(\alpha_k+1) \mod p \text{ if all $\beta_j$ are even}.
\end{array} \right.
\end{displaymath}

 Using these properties, we may easily deduce that  $r_2(n)$ is a multiplicative function of $n$. 
 %Or, in order to study the the subword complexity of $r_2(n)_{n\geq 0}$, it would be necessary to find some additive properties or 
 Recall that we would like to study the subword complexity of $r_2(n)_{n\geq 0}$, that is the number of distinct factors of form $r_2(j)r_2(j+1)\cdots r_2(j+m-1)$, when $j\in \N$. Hence, it would be useful to describe some additive properties of $r_2(n)_{n\geq 0}$; for instance, it would be interesting to find some relations between  $r_2(j+N)$ and $r_2(j)$, for some positive integers $j, N$. This seems to be a rather difficult question about which we are not able to say anything conclusive.

 \section{Conclusion}

  It would be also interesting to investigate  the following general question.\\
 
 \emph{Is it true that Carlitz's analogs of classical constants all have a ``low'' complexity (i.e., polynomial or subexponential)?} 
\\

 The first clue in this direction are the examples provided by Theorems \ref{algebraic}, \ref{case2} and \ref{caseq}. 
 \medskip 
Notice also that a positive answer would reinforce the  differences between $\mathbb R$ and $\F_q((T^{-1}))$ %%b 
as hinted in our introduction.  
 When investigating these problems, we need, in general, the Laurent series expansions of such functions. In this context, one has to mention  the work of Berth� \cite{Berthe,  Berthe_2, Berthe_3, Berthe_4}, where 
 %%b
 some Laurent series expansions of  Carlitz's functions are described. 

\medskip 

%%b 
When a Laurent series has a ``low'' complexity,  the combinatorial structure of its sequence of coefficients is rich and this can be used to derive  some  interesting Diophantine properties. Using this principle, bounds for irrationality measures can be obtained for elements of the class of  Laurent series with at most linear complexity.

\section*{Acknowledgements}
I would like to express my gratitude to my advisor Boris Adamczewski for his very valuable comments, suggestions and corrections during the preparation of this article.
I also thank 
%%b 
Jean-Paul Allouche for sending me his unpublished note concerning the problem studied in Section 3,  and  Julien Cassaigne for suggesting  the idea used in Remark \ref{saturate_eq}.

\end{document}